%% file: kernel-regression.tex
\documentclass[11pt]{article}

\usepackage{macros}
\usepackage{psfrag}
\usepackage{multirow}
\usepackage[numbers]{natbib}
\usepackage{fullpage}

\long\def\comment#1{}

\makeatletter
\long\def\@makecaption#1#2{
  \vskip 0.8ex
  \setbox\@tempboxa\hbox{\small {\bf #1:} #2}
  \parindent 1.5em  %% How can we use the global value of this???
  \dimen0=\hsize
  \advance\dimen0 by -3em
  \ifdim \wd\@tempboxa >\dimen0
  \hbox to \hsize{
    \parindent 0em
    \hfil
    \parbox{\dimen0}{\def\baselinestretch{0.96}\small
      {\bf #1.} #2
    }
    \hfil}
  \else \hbox to \hsize{\hfil \box\@tempboxa \hfil}
  \fi
}
\makeatother

\begin{document}

\begin{center}

  {\bf{\LARGE{
        Divide and Conquer Kernel Ridge Regression: \\ A Distributed Algorithm
        with Minimax Optimal Rates}}}

  \vspace*{.3in}

  {\large{
      \begin{tabular}{ccc}
	Yuchen Zhang $^1$ & John C.\ Duchi$^1$ & Martin J. Wainwright$^{1,2}$ \\
        {\texttt{yuczhang@eecs.berkeley.edu}} &
        \texttt{jduchi@eecs.berkeley.edu} &
        \texttt{wainwrig@stat.berkeley.edu}
      \end{tabular}
  }}

  \vspace*{.2in}

  \begin{tabular}{ccc}
    Department of Statistics$^1$ & & Department of Electrical Engineering
    and Computer Science$^2$
  \end{tabular}

  \begin{tabular}{c}
    UC Berkeley,
    Berkeley, CA 94720
  \end{tabular}

  \vspace*{.2in}

  May 2013

\end{center}

\begin{abstract}
  We establish optimal convergence rates for a decomposition-based scalable
  approach to kernel ridge regression.  The method is simple to describe: it
  randomly partitions a dataset of size $\totalobs$ into $\nummac$ subsets of
  equal size, computes an independent kernel ridge regression estimator for
  each subset, then averages the local solutions into a global predictor.
  This partitioning leads to a substantial reduction in computation time
  versus the standard approach of performing kernel ridge regression on all
  $\totalobs$ samples.  Our two main theorems establish that despite the
  computational speed-up, statistical optimality is retained: as long as
  $\nummac$ is not too large, the partition-based estimator achieves the
  statistical minimax rate over all estimators using the set of $\totalobs$
  samples.  As concrete examples, our theory guarantees that the number of
  processors $\nummac$ may grow nearly linearly for finite-rank kernels and
  Gaussian kernels and polynomially in $\totalobs$ for Sobolev spaces, which
  in turn allows for substantial reductions in computational cost. We conclude
  with experiments on both simulated data and a music-prediction task that
  complement our theoretical results, exhibiting the computational and
  statistical benefits of our approach.
\end{abstract}

\input{introduction}
\input{background}
\input{main-results}
\input{major-proofs}
\input{oracle-proofs}
\input{simulations}
\input{year-prediction-experiment}
\input{discussion}

\input{acknowledgements}
\appendix

\input{bias-proof}
\input{variance-proof}

\input{oracle-variance-proof}
\input{oracle-bias-proof}

\bibliographystyle{abbrvnat}
\bibliography{bib}

\end{document}

%% file: introduction.tex
\section{Introduction}

In non-parametric regression, the statistician receives $\totalobs$
samples of the form $\{(x_i, y_i)\}_{i=1}^\totalobs$, where each $x_i
\in \XC$ is a covariate and $y_i \in \R$ is a real-valued response,
and the samples are drawn i.i.d.\ from some unknown joint distribution
$\P$ over $\XC \times \R$.  The goal is to estimate a function
$\fapprox: \XC \rightarrow \R$ that can be used to predict future
responses based on observing only the covariates.  Frequently, the
quality of an estimate $\fapprox$ is measured in terms of the
mean-squared prediction error $\E[(\fapprox(X) - Y)^2]$, in which case
the conditional expectation $\fopt(x) = \E[Y \mid X =x]$ is optimal.
The problem of non-parametric regression is a classical one, and a
researchers have studied a wide range of estimators (see, e.g., the
books~\cite{GyorfiEtal,Wasserman2006,vandeGeer}).  One class of
methods, known as regularized $M$-estimators~\cite{vandeGeer}, are
based on minimizing the combination of a data-dependent loss function
with a regularization term.  The focus of this paper is a popular
$M$-estimator that combines the least-squares loss with a
squared Hilbert norm penalty for regularization.  When working in a
reproducing kernel Hilbert space (RKHS), the resulting method is known
as \emph{kernel ridge regression}, and is widely used in
practice~\citep{Hastie2001, Shawe-Taylor2004}.  Past work has
established bounds on the estimation error for RKHS-based
methods~\citep[e.g.,][]{Koltchinskii2006,Mendelson02,vandeGeer,Zhang2005b},
which have been refined and extended in more recent
\mbox{work~\citep[e.g.,][]{Steinwart2009}.}

Although the statistical aspects of kernel ridge regression (KRR) are
well-understood, the computation of the KRR estimate can be
challenging for large datasets.  In a standard
implementation~\citep{Saunders1998}, the kernel matrix must be
inverted, which requires costs $\order(\overallsize^3)$ and
$\order(\overallsize^2)$ in time and memory respectively.  Such
scalings are prohibitive when the sample size $\totalobs$ is large.
As a consequence, approximations have been designed to avoid the
expense of finding an exact minimizer. One family of approaches is
based on low-rank approximation of the kernel matrix; examples include
kernel PCA~\citep{Scholkopf1998b}, the incomplete Cholesky
decomposition~\citep{Fine2002}, or Nystr\"om
sampling~\citep{Williams2001}. These methods reduce the time
complexity to $\order(d\overallsize^2)$ or $\order(d^2\overallsize)$,
where $d\ll \overallsize$ is the preserved rank.  The associated
prediction error has only been studied very recently.  Concurrent work
by~\citet{Bach2013} establishes conditions on the maintained rank that
still guarantee optimal convergence rates; see the discussion for more
detail.
%
% To our knowledge, however, there are no results showing that such low-rank
% versions of KRR still achieve minimax-optimal rates in estimation error.
%
A second line of research has considered
early-stopping of iterative optimization algorithms for KRR, including
gradient descent~\citep{Yao2007,Raskutti2011} and conjugate gradient
methods~\citep{Blanchard2010}, where early-stopping provides
regularization against over-fitting and improves run-time. If the
algorithm stops after $t$ iterations, the aggregate time complexity is
$\order(t \overallsize^2)$.

In this work, we study a different decomposition-based
approach.  The algorithm is appealing in
its simplicity: we partition the dataset of size $\overallsize$
randomly into $\nummac$ equal sized subsets, and we compute the
kernel ridge regression estimate $\fapprox_i$ for each of the $i = 1,
\ldots, \nummac$ subsets independently, with a \emph{careful choice}
of the regularization parameter. The estimates are then averaged via
$\funcavg = (1 / \nummac) \sum_{i=1}^\nummac \fapprox_i$.  Our main
theoretical result gives conditions under which the average $\funcavg$
achieves the minimax rate of convergence over the underlying Hilbert
space.
%
%% In particular, we prove that as long as $\nummac$ is below a certain
%% threshold, which in most cases ia a polynomial of $\overallsize$, then
%% $\funcavg$ achieves the optimal rate of convergence. For certain types of
%% kernels, including the finite rank kernels or Gaussian kernels, $\nummac$
%% can increase nearly linearly with $\overallsize$.
%
Even using naive implementations of KRR, this decomposition gives time
and memory complexity scaling as $\order(\totalobs^3 / \nummac^2)$ and
$\order(\totalobs^2 / \nummac^2)$, respectively. Moreover, our
approach dovetails naturally with parallel and distributed
computation: we are guaranteed superlinear speedup with $\nummac$
parallel processors (though we must still communicate the function
estimates from each processor).
Divide-and-conquer approaches have been studied by several
authors, including \citet{McDonald2010} for perceptron-based
algorithms, \citet{Kleiner2012} in distributed versions of the
bootstrap, and \citet{Zhang2012a} for parametric smooth convex
optimization problems.  This paper demonstrates the potential benefits
of divide-and-conquer approaches for nonparametric and
infinite-dimensional regression problems.

One difficulty in solving each of the sub-problems independently is
how to choose the regularization parameter.  Due to the
infinite-dimensional nature of non-parametric problems, the choice of
regularization parameter must be made with
care~\citep[e.g.,][]{Hastie2001}.  An interesting consequence of our
theoretical analysis is in demonstrating that, even though each
partitioned sub-problem is based only on the fraction
$\totalobs/\nummac$ of samples, it is nonetheless \emph{essential to
  regularize the partitioned sub-problems as though they had all
  $\totalobs$ samples}.  Consequently, from a local point of view,
each sub-problem is under-regularized.  This %
``under-regularization'' allows the bias of each local estimate to be
very small, but it causes a detrimental blow-up in the variance.
However, as we prove, the $\nummac$-fold averaging underlying the
method reduces variance enough that the resulting estimator $\funcavg$
still attains optimal convergence rate.
%\footnote{As a simple example of
%  this interplay, consider estimating a density functional $f$ on $[0,
%    1]$ using a histogram~\citep[Chapter 6.2]{Wasserman2006}. The
%  optimal bin width for the histogram is $\totalnumobs^{-1/3}$, which
%  yields a convergence rate of $\totalnumobs^{-2/3}$ in $L^2([0,
%    1])$. Even when the problems are partitioned, it is clear that
%  having each partition use bins of width $\totalnumobs^{-1/3}$, then
%  averaging the resulting counts, yields an optimal estimator (in
%  fact, the same as in the non-partitioned case).}
%% TOO TANGENTIAL

The remainder of this paper is organized as follows.  We begin in
Section~\ref{sec:background} by providing background on the kernel
ridge regression estimate and discussing the assumptions that underlie
our analysis.  In Section~\ref{sec:main-result}, we present our main
theorems on the mean-squared error between the averaged estimate
$\funcavg$ and the optimal regression function $\fopt$.  We provide
both a result when the regression function $\fopt$ belongs to the
Hilbert space $\hilbertspace$ associated with the kernel, as well as a
more general oracle inequality that holds for a general $\fopt$.  We
then provide several corollaries that exhibit concrete consequences of
the results, including convergence rates of $\kerrank / \totalnumobs$
for kernels with finite rank $\kerrank$, and convergence rates of
$\totalnumobs^{-2 \smoothness / (2 \smoothness + 1)}$ for estimation
of functionals in a Sobolev space with $\smoothness$-degrees of
smoothness.  As we discuss, both of these estimation rates are
minimax-optimal and hence unimprovable.
We devote Sections~\ref{sec:proof-results-and-corollaries}
and~\ref{sec:proof-oracle} to the proofs of our results, deferring more
technical aspects of the analysis to
appendices. Lastly, we present simulation results in
Section~\ref{sec:experiments} to further explore our theoretical
results, while Section~\ref{sec:real-data} contains experiments with a
reasonably large music prediction experiment.

% Local Variables:
% TeX-master: "kernel-regression"
% End:

%% file: background.tex
\section{Background and problem formulation}
\label{sec:background}

We begin with the background and notation required for a precise
statement of our problem.

\subsection{Reproducing kernels}

The method of kernel ridge regression is based on the idea of a
reproducing kernel Hilbert space.  We provide only a very brief
coverage of the basics here, referring the reader to one of the many books
on the topic
(e.g.,~\cite{Wahba,Shawe-Taylor2004,Berlinet2004,Gu2002}) for further
details.  Any symmetric and positive semidefinite kernel function
$\kernel: \XC \times \XC\to \R$ defines a reproducing kernel Hilbert
space (RKHS for short).  For a given distribution $\P$ on $\XC$, the
Hilbert space is strictly contained within $L^2(\P)$.  For each $x \in
\XC$, the function $z \mapsto \kernel(z, x)$ is contained with the
Hilbert space $\HilbertSpace$; moreover, the Hilbert space is endowed
with an inner product $\<\cdot, \cdot\>_\hilbertspace$ such that
$\kernel(\cdot, x)$ acts as the representer of evaluation, meaning
\begin{equation}
  \label{eqn:reproducing-kernel-relation}
  \<f, \kernel(x,\cdot )\>_\hilbertspace = f(x) ~~~ \mbox{for~} f
  \in \hilbertspace.
\end{equation}
We let $\hnorm{g} \defeq \sqrt{\<g, g\>_\hilbertspace}$ denote the
norm in $\hilbertspace$, and similarly $\ltwo{g} \defeq (\int_\XC
g(x)^2 d\P(x))^{1/2}$ denotes the norm in $\lsspace$.  Under suitable
regularity conditions, Mercer's theorem guarantees that the kernel has
an eigen-expansion of the form
\begin{equation*}
  K(x,x') = \sum_{j=1}^\infty \eigenvalue_j \base_j(x)\base_j(x'),
\end{equation*}
where $\eigenvalue_1 \geq \eigenvalue_2 \geq \dots \geq 0$ are a
non-negative sequence of eigenvalues, and $\{\base_j\}_{j=1}^\infty$
is an orthonormal basis
for $\lsspace$.

From the reproducing relation~\eqref{eqn:reproducing-kernel-relation},
we have $\langle \base_j, \base_j \rangle_\hilbertspace =
1/\eigenvalue_j$ for any $j$ and $\langle \base_j, \base_{j'}
\rangle_\hilbertspace = 0$ for any $j \neq j'$. For any $f
\in \HilbertSpace$, by defining the basis coefficients $\theta_j =
\inprod{f}{\base_j}_{L^2(\P)}$ for $j = 1, 2, \ldots$, we can expand
the function in terms of these coefficients as $f = \sum_{j=1}^\infty
\theta_j \basis_j$, and simple calculations show that
\begin{equation*}
  \ltwo{f}^2 = \int_\XC f^2(x) d \P(x) = \sum_{j=1}^\infty \theta_j^2,
  ~~~ \mbox{and} ~~~ \hnorm{f}^2 = \<f, f\> = \sum_{j=1}^\infty
  \frac{\theta_j^2}{\eigenvalue_j}.
\end{equation*}
Consequently, we see that the RKHS can be viewed as an elliptical subset
of the sequence space $\ell^2(\mathbb{N})$ as defined by the non-negative
eigenvalues $\{\eigenvalue_j\}_{j=1}^\infty$.

%%%%%%%%%%%%%%%%%%%%%%%%%%%%%%%%%%%%%%%%%%%%%%%%%%%%%%%%%%%%%%%%%%%%%%%%%

\subsection{Kernel ridge regression}

Suppose that we are given a data set
$\{(x_i,y_i)\}_{i=1}^\totalnumobs$ consisting of $\totalnumobs$
i.i.d.\ samples drawn from an unknown distribution $\P$ over $\Xspace
\times \real$, and our goal is to estimate the function that minimizes
the mean-squared error $\Exs[(f(X) - Y)^2]$, where the expectation is
taken jointly over $(X,Y)$ pairs.  It is well-known that the optimal
function is the conditional mean $\fstar(x) \defn \Exs[Y \mid X = x]$.
In order to estimate the unknown \mbox{function $\fstar$,} we consider
an $M$-estimator that is based on minimizing a combination of the
least-squares loss defined over the dataset with a weighted penalty
based on the squared Hilbert norm,
\begin{align}
  \label{eqn:krr-estimate}
  \funcapprox & \defeq \argmin_{f \in \HilbertSpace} \bigg \{
\frac{1}{\overallsize} \sum_{i=1}^\overallsize \left(f(x_i) -
y_i\right)^2 + \lambda \hnorm{f}^2 \bigg\},
\end{align}
where $\lambda > 0$ is a regularization parameter.  When
$\hilbertspace$ is a reproducing kernel Hilbert space, then the
estimator~\eqref{eqn:krr-estimate} is known as the \emph{kernel ridge
  regression estimate}, or KRR for short. It is a natural
generalization of the ordinary ridge regression
estimate~\cite{Hoerl70} to the non-parametric setting.

By the representer theorem for reproducing kernel Hilbert
spaces~\cite{Wahba}, any solution to the KRR
program~\eqref{eqn:krr-estimate} must belong to the linear span of the
kernel functions $\{\kernel(\cdot, x_i), \, i = 1, \ldots,
\totalnumobs\}$.  This fact allows the computation of the KRR estimate to
be reduced to an $\totalnumobs$-dimensional quadratic program,
involving the $\totalnumobs^2$ entries of the kernel matrix $\{
\kernel(x_i, x_j), \; i,j = 1, \ldots, \numobs\}$.  On the statistical
side, a line of past
work~\citep{vandeGeer,Zhang2005b,Caponnetto2007,Steinwart2009,Hsu2012} has
provided bounds on the estimation error of $\fapprox$ as a function of
$\totalnumobs$ and $\lambda$.

% Local Variables:
% TeX-master: "kernel-regression.tex"
% End:

%% file: main-results.tex
\section{Main results and their consequences}
\label{sec:main-result}

We now turn to the description of our algorithm, followed by the
statements of our main results, namely
Theorems~\ref{theorem:error-bound} and~\ref{theorem:oracle}.  Each
theorem provides an upper bound on the mean-squared prediction error
for any trace class kernel.  The second theorem is of ``oracle
type,'' meaning that it applies even when the true regression function
$\fstar$ does not belong to the Hilbert space $\HilbertSpace$, and
hence involves a combination of approximation and estimation error
terms.  The first theorem requires that $\fstar \in \HilbertSpace$,
and provides somewhat sharper bounds on the estimation error in this
case.  Both of these theorems apply to any trace class kernel, but as
we illustrate, they provide concrete results when applied to specific
classes of kernels.  Indeed, as a corollary, we establish that our
distributed KRR algorithm achieves the statistically minimax-optimal
rates for three different kernel classes, namely finite-rank, Gaussian
and Sobolev.

%%%%%%%%%%%%%%%%%%%%%%%%%%%%%%%%%%%%%%%%%%%%%%%%%%%%%%%%%%%%%%%%%%%%%%%%%%%%%%

\subsection{Algorithm and assumptions}

The divide-and-conquer algorithm \algname\ is easy to describe. We are
given $\overallsize$ samples drawn i.i.d.\ according to the
distribution $\P$. Rather than solving the kernel ridge regression
problem~\eqref{eqn:krr-estimate} on all $\totalobs$ samples, the
\algname\ method executes the following three steps:
\begin{enumerate}
  \item Divide the set of samples $\{(x_1, y_1), \ldots, (x_\totalobs,
    y_\totalobs)\}$ evenly and uniformly at randomly into the $\nmachine$
    disjoint subsets $S_1, \ldots, S_\nmachine \subset \XC \times \R$.
  \item For each $i = 1, 2, \ldots, \nmachine$, compute the
    \emph{local KRR estimate}
    \begin{align}
      \label{eqn:sub-krr}
      \funcapprox_i \defeq \argmin_{f \in \HilbertSpace} \bigg\{
      \frac{1}{|S_i|} \sum_{(x,y)\in S_i} \left(f(x)-y\right)^2 +
      \lambda \hnorm{f}^2 \bigg\}.
    \end{align}
  \item \label{item:compute-average} Average together the local estimates and
    output $\funcavg = \frac{1}{\nmachine}\sum_{i=1}^\nmachine \funcapprox_i$.
\end{enumerate}
This description actually provides a family of estimators, one for each choice
of the regularization parameter $\lambda > 0$.  Our main result applies to any
choice of $\lambda$, while our corollaries for specific
kernel classes optimize $\lambda$ as a function of the kernel.

We now describe our main assumptions.  Our first assumption, for which
we have two variants, deals with the tail behavior of the basis
functions $\{\basis_j\}_{j=1}^\infty$.
\newcounter{savetwoassumption}
\setcounter{savetwoassumption}{\value{assumption}}
\begin{assumption}
  \label{assumption:kernel}
  For some $\moment \ge 2$, there is a constant $\momentbound <
  \infty$ such that $\E[\basis_j(X)^{2\moment}] \le
  \momentbound^{2\moment}$ for all $j = 1, 2, \ldots$.
\end{assumption}

%%%%%%%%%%%%%%%%%%%%%%%%%%%%%%%%%%%%%%%%%%%%%%%%%%%%%%%%%%%%%%%%%%%%%%%%%%

\setcounter{assumption}{\value{savetwoassumption}}
\renewcommand{\theassumption}{\Alph{assumption}$^\prime$}

\noindent In certain cases, we show that sharper error guarantees can
be obtained by enforcing a stronger condition of uniform boundedness:
\begin{assumption}
  \label{assumption:bounded-kernel}
  There is a constant $\momentbound < \infty$ such that $\sup_{x \in
    \XC} |\basis_j(x)| \le \momentbound$ for all $j = 1, 2, \ldots$.
\end{assumption}
\renewcommand{\theassumption}{\Alph{assumption}}

%%%%%%%%%%%%%%%%%%%%%%%%%%%%%%%%%%%%%%%%%%%%%%%%%%%%%%%%%%%%%%%%%%%%%%%%%

Recalling that $\fstar(x) \defn \Exs[Y \, \mid X = x]$, our second
assumption involves the deviations of the zero-mean noise variables $Y
- \fstar(x)$.  In the simplest case, when $\fstar \in \HilbertSpace$,
we require only a bounded variance condition:
\newcounter{saveassumption}
\setcounter{saveassumption}{\value{assumption}}
\begin{assumption}
  \label{assu:function-space}
  The function $\functrue \in \HilbertSpace$, and for $x \in \XC$,
  we have $\E[(Y - \functrue(x))^2 \mid x] \le \stddev^2$.
\end{assumption}

\noindent When the function $\fopt \not \in \hilbertspace$, we require
a slightly stronger variant of this assumption.  For each $\lambdabase
\ge 0$, define
\begin{align}
  \label{eqn:new-pop-objective-function}
  \foptimalreg = \argmin_{f\in \HilbertSpace} \left\{\E\left[(f(X) -
    Y)^2\right] + \lambdabase \hnorm{f}^2\right\}.
\end{align}
Note that $\fstar = \fstar_0$ corresponds to the usual regression
function, though the infimum may not be attained for
$\lambdabase = 0$ (our analysis addresses such cases).
Since $\fstar \in L^2(\P)$, we are guaranteed that for each
$\lambdabase \geq 0$, the associated mean-squared error $\varreg^2(x)
\defeq \E[(y - \foptimalreg(x))^2 \mid x]$ is finite for almost every
$x$.  In this more general setting, the following assumption replaces
Assumption~\ref{assu:function-space}:

\setcounter{assumption}{\value{saveassumption}}
\renewcommand{\theassumption}{\Alph{assumption}$^\prime$}

\begin{assumption}
  \label{assumption:response-moment-bound}
  For any $\lambdabase \geq 0$, there exists a constant $\altvarreg <
  \infty$ such that $\altvarreg^4 = \E[\varreg^4(X)]$.
\end{assumption}
\renewcommand{\theassumption}{\Alph{assumption}}
This
condition with $\lambdabase = 0$ is slightly stronger than
Assumption~\ref{assu:function-space}.

%%%%%%%%%%%%%%%%%%%%%%%%%%%%%%%%%%%%%%%%%%%%%%%%%%%%%%%%%%%%%%%%%%%%%%%%%

\subsection{Statement of main results}

With these assumptions in place, we are now ready for the statements
of our main results.  All of our results give bounds on the
mean-squared estimation error $\Exs[\|\bar{f} - \fstar\|_2^2]$
associated with the averaged estimate $\bar{f}$ based on an assigning
$\numobs = \totalobs/\nummac$ samples to each of $\nummac$ machines.
Both theorem statements involve the following three kernel-related
quantities:
\begin{align}
  \label{eqn:define-shorthands}
  \tr(K) \defeq \sum_{j=1}^\infty\eigenvalue_j, ~~~
  \fulltracelambda \defeq
  \sum_{j=1}^\infty \frac{1}{1 + \lambda/\eigenvalue_j},
  ~~~\mbox{and}~~~
  \tailsum = \sum_{j = d+1}^\infty \eigenvalue_j.
\end{align}
The first quantity is the kernel trace, which serves a crude estimate
of the ``size'' of the kernel operator, and assumed to be finite.  The
second quantity $\fulltracelambda$, familiar from previous work on
kernel regression~\citep{Zhang2005b}, is known as the ``effective
dimensionality'' of the kernel $K$ with respect to $\lsspace$.
Finally, the quantity $\tailsum$ is parameterized by a positive
integer $d$ that we may choose in applying the bounds, and it
describes the tail decay of the eigenvalues of $K$.  For $d = 0$, note
that $\beta_0$ reduces to the ordinary trace.  Finally, both theorems
involve one further quantity that depends on the number of moments
$\moment$ in Assumption~\ref{assumption:kernel}, namely
\begin{equation}
  \maxlog \defeq \max\left\{\sqrt{\max\{\moment, \log(d)\}},
  \frac{\max\{\moment, \log(d)\}}{\numobs^{1/2 - 1/\moment}}\right\}.
  \label{eqn:define-maxlog}
\end{equation}
Here the parameter $d \in \N$ is a quantity that may be optimized
to obtain the sharpest possible upper bound and may be chosen arbitrarily.
(The algorithm's execution is independent of $d$.)

\begin{theorem}
  \label{theorem:error-bound}
  With $\fstar \in \HilbertSpace$ and under
  Assumptions~\ref{assumption:kernel} and~\ref{assu:function-space}, the
  mean-squared error of the averaged estimate $\bar{f}$ is upper bounded
  as
  \begin{align}
    \label{EqnGoodUpper}
    \E\left[\ltwo{\bar{f} - \fopt}^2\right] & \leq \left(8 +
    \frac{12}{\nummac}\right) \lambda \hnorm{\fopt}^2 + \frac{12
      \stddev^2 \fulltracelambda}{\totalobs} + \inf_{d \in \N}
    \big\{ \Term_1(d) + \Term_2(d) + \Term_3(d) \big\},
    % \bigg\{\sum_{j=1}^3 T_j(d) \bigg\}.
  \end{align}
  where
  \begin{align*}
    \Term_1(d) & = \frac{8 \momentbound^4 \hnorm{\fopt}^2 \tr(K)
      \tailsum}{\lambda}, \quad \Term_2(d) = \frac{4 \hnorm{\fopt}^2 + 2
      \stddev^2 / \lambda}{m} \left(\eigenvalue_{d + 1} + \frac{12
      \momentbound^4 \tr(K) \tailsum}{\lambda} \right), \quad \mbox{and}
    \\
    \Term_3(d) & = \left(C \maxlog \frac{\momentbound^2
      \fulltracelambda}{\sqrt{\numobs}} \right)^k \ltwo{\fopt}^2 \left(1 +
    \frac{2 \stddev^2}{\nummac \lambda} + \frac{4
      \hnorm{\fopt}^2}{\nummac}\right),
  \end{align*}
  and $C$ denotes a universal (numerical) constant.
\end{theorem}

Theorem~\ref{theorem:error-bound} is a general result that applies to
any trace-class kernel.  Although the statement appears somewhat
complicated at first sight, it yields concrete and interpretable
guarantees on the error when specialized to particular kernels, as we
illustrate in Section~\ref{sec:corollaries}.

Before doing so, let us provide a few heuristic arguments for
intuition.  In typical settings, the term $T_3(d)$ goes
to zero quickly: if the number of moments $\moment$ is large and
number of partitions $\nummac$ is small---say enough to guarantee that
$(\fulltracelambda^2 \totalobs / \nummac)^{-k/2} = \order(1 /
\totalobs)$---it will be of lower order. As for the remaining terms,
at a high level, we show that an appropriate choice of the free
parameter $d$ leaves the first two terms in the upper
bound~\eqref{EqnGoodUpper} dominant.  Note that the terms
$\eigenvalue_{d+1}$ and $\tailsum$ are decreasing in $d$ while the
term $\maxlog$ increases with $d$. However, the increasing term
$\maxlog$ grows only logarithmically in $d$, which allows us to choose
a fairly large value without a significant penalty.  As we show in our
corollaries, for many kernels of interest, as long as the number of
machines $\nummachine$ is not ``too large,'' this tradeoff is such
that $\Term_1(d)$ and $\Term_2(d)$ are also of lower order compared to
the two first terms in the bound~\eqref{EqnGoodUpper}.  In such
settings, Theorem~\ref{theorem:error-bound} guarantees an upper bound
of the form
\begin{equation}
  \E\left[\ltwo{\bar{f} - \fopt}^2\right] = \order \Big(
  \underbrace{\lambda \hnorm{\fopt}^2}_{\mbox{\scriptsize{Squared
        bias}}} \; + \; \underbrace{\frac{\stddev^2
      \fulltracelambda}{\totalobs}}_{\mbox{\scriptsize{Variance}}}
  \Big).
  \label{eqn:competing-terms}
\end{equation}
This inequality reveals the usual bias-variance trade-off in
non-parametric regression; choosing a smaller value of $\lambda > 0$
reduces the first squared bias term, but increases the second variance
term.  Consequently, the setting of $\lambda$ that minimizes the sum
of these two terms is defined by the relationship
\begin{align}
  \label{EqnOptimalLambda}
  \lambda \hnorm{\fopt}^2 & \; \simeq \; \stddev^2
  \frac{\fulltracelambda}{\totalobs}.
\end{align}
This type of fixed point equation is familiar from work on oracle inequalities
and local complexity measures in empirical process
theory~\citep{Bartlett2005,Koltchinskii2006, vandeGeer,Zhang2005b}, and when
$\lambda$ is chosen so that the fixed point equation~\eqref{EqnOptimalLambda}
holds this (typically) yields minimax optimal convergence
rates~\cite{Bartlett2005,Koltchinskii2006,Zhang2005b,Caponnetto2007}.
In Section~\ref{sec:corollaries}, we provide detailed examples in which
the choice $\lambda^*$ specified by equation~\eqref{EqnOptimalLambda},
followed by application of Theorem~\ref{theorem:error-bound}, yields
minimax-optimal prediction error (for the \algname\ algorithm)
for many kernel classes. \\

\vspace*{.05in}

We now turn to an error bound that applies without requiring that
$\fopt \in \hilbertspace$.  Define the radius $\radius =
\hnorm{\foptimalreg}$, where the population regression function
$\foptimalreg$ was previously
defined~\eqref{eqn:new-pop-objective-function}.  The theorem requires
a few additional conditions to those in
Theorem~\ref{theorem:error-bound}, involving the quantities
$\tr(K)$, $\fulltracelambda$ and $\tailsum$
defined in Eq.~\eqref{eqn:define-shorthands}, as well as the error moment
$\altvarreg$ from Assumption~\ref{assumption:response-moment-bound}.
We assume that the triplet $(\nmachine, d, \moment)$ of positive
integers satisfy the conditions
\begin{equation}
  \label{eqn:oracle-condition}
  \begin{split}
    & \tailsum \le \frac{\lambda}{
      (\radius^2 + \altvarreg^2/\lambda)\overallsize},
    ~~~~ \eigenvalue_{d + 1} \le
    \frac{1}{(\radius^2 + \altvarreg^2 / \lambda) \totalnumobs}, \\
    %
    %& \frac{(\momentbound^2 \log(d)\fulltracelambda)^2}{\overallsize} \leq
    %  \frac{1}{\nmachine} ~~~~\mbox{and}~~~~ \left(\maxlog
    %  \frac{\momentbound^2 \fulltracelambda}{
    %    \sqrt{\overallsize/\nmachine}}\right)^\moment \leq
    %  \frac{1}{(\radius^2 + \altvarreg^2/\lambda)\overallsize}.
    & \nmachine \leq \min\left\{
    \frac{\sqrt{\overallsize}}{\momentbound^2\fulltracelambda \log(d)},
    \frac{\overallsize^{1 - \frac{2}{\moment}}}{
      (\radius^2 + \altvarreg^2/\lambda)^{2/k}
      (\maxlog \momentbound^2\fulltracelambda)^2}\right\}.
  \end{split}
\end{equation}
We then have the following result:
\begin{theorem}
\label{theorem:oracle}
Under condition~\eqref{eqn:oracle-condition},
Assumption~\ref{assumption:kernel} with $\moment \geq 4$, and
Assumption~\ref{assumption:response-moment-bound}, for any $\lambda
\geq \lambdabase$ and $q > 0$ we have
\begin{align}
\label{EqnOracleBound}
\E \left[ \ltwo{\favg-\foptimal}^2 \right] & \leq \left(1 +
\frac{1}{q}\right) \inf_{\hnorm{f} \leq \radius} \ltwo{ f -
  \foptimal}^2 + (1 + q) \, \RES
\end{align}
where the residual term is given by
\begin{align}
\label{EqnDefnResidual}
\RES & \defn \bigg(\Big (4 + \frac{C}{\nummac} \Big)
(\lambda-\lambdabase) \radius^2 + \frac{C \fulltracelambda
  \momentbound^2 \altvarreg^2 }{\overallsize} + \frac{C}{\totalnumobs}
\bigg),
\end{align}
and $C$ denotes a universal (numerical) constant.
\end{theorem}

\paragraph{Remarks:}  Theorem~\ref{theorem:oracle} is an instance of an
oracle inequality, since it upper bounds the mean-squared error in
terms of the error $\inf \limits_{\|f\|_\HilbertSpace \leq R} \|f -
\fstar\|_2^2$, which may only be obtained by an oracle knowing the
sampling distribution $\P$, plus the residual error
term~\eqref{EqnDefnResidual}.

In some situations, it may be difficult to verify
Assumption~\ref{assumption:response-moment-bound}.  In such scenarios,
an alternate condition suffices.  For instance, if there exists a
constant $\ymoment < \infty$ such that $\E[Y^4]\leq \ymoment^4$, then
the bound~\eqref{EqnOracleBound} holds (under condition~\eqref{eqn:oracle-condition}) with $\altvarreg^2$ replaced by
$\sqrt{8\tr(K)^2 \radius^4 \momentbound^4 + 8\ymoment^4}$---that is,
with the alternative residual error
\begin{align}
\label{EqnDefnResidualAlt}
\ALTRES & \defn \bigg(\Big (2 + \frac{C}{\nummac} \Big)
(\lambda-\lambdabase) \radius^2 + \frac{C \fulltracelambda
  \momentbound^2 \sqrt{8\tr(K)^2 \radius^4 \momentbound^4 +
    8\ymoment^4} }{\overallsize} + \frac{C}{\totalnumobs} \bigg).
\end{align}
In essence, if the response variable $Y$ has sufficiently many moments, the
prediction mean-square error $\altvarreg^2$ in the statement of
Theorem~\ref{theorem:oracle} can be replaced constants related to the size of
$\hnorm{\foptimalreg}$. See Section~\ref{sec:proof-residual-alternate} for a
proof of inequality~\eqref{EqnDefnResidualAlt}. % \\

% \vspace*{.1in}

In comparison with Theorem~\ref{theorem:error-bound},
Theorem~\ref{theorem:oracle} provides somewhat looser bounds. It is,
however, instructive to consider a few special cases. For the first,
we may assume that $\fopt \in \HilbertSpace$, in which case
$\hnorm{\fopt} < \infty$. In this setting, the choice $\lambdabase = 0$
(essentially) recovers Theorem~\ref{theorem:error-bound}, since there
is no approximation error. Taking $q \rightarrow 0$, we are thus left
with the bound
\begin{align}
\Exs\|\favg - \fstar\|_2^2] & \; \precsim \; \lambda \,
  \hnorm{\fopt}^2 + \frac{\fulltracelambda\momentbound^2\tau_0^2
  }{\overallsize},
\end{align}
where $\precsim$ denotes an inequality up to constants.  By
inspection, this bound is roughly equivalent to
Theorem~\ref{theorem:error-bound}; see in particular the
decomposition~\eqref{eqn:competing-terms}.
On the other hand, when the condition $\fopt \in \HilbertSpace$ fails
to hold, we can take $\lambdabase =\lambda$, and then choose $q$ and
$\lambdabase$ to balance between the familiar approximation and
estimation errors. In this case, we have
\begin{align}
\Exs[\|\fbar - \fstar\|_2^2] & \precsim \left(1 + \frac{1}{q}\right) \;
\underbrace{\inf_{\hnorm{f} \le \radius} \ltwo{f - \fopt}^2}_{
  \mbox{approximation}} ~+~ (1 + q) \underbrace{\left( \frac{
    \fulltracelambda \momentbound^2 \altvarreg^2}{\totalnumobs}
  \right)}_{ \mbox{estimation}}.
\end{align}

The condition~\eqref{eqn:oracle-condition} required to apply
Theorem~\ref{theorem:oracle} imposes constraints on the number
$\nummac$ of subsampled data sets that are stronger than
those of Theorem~\ref{theorem:error-bound}. In particular, when
ignoring constants and logarithm terms, the quantity $\nummac$
may grow at rate $\sqrt{\totalnumobs/\fulltracelambdasq}$.
By contrast, Theorem~\ref{theorem:error-bound} allows $\nummac$
to grow as quickly as $\totalnumobs/\fulltracelambdasq$ (recall
the remarks on $\Term_3(d)$ following Theorem~\ref{theorem:error-bound} or
look ahead
to condition~\eqref{eqn:general-num-machines}). Thus---at least
in our current analysis---generalizing to the case that
$\fopt \not\in \hilbertspace$ prevents
us from dividing the data into finer subsets.

%In particular, the
%quantity $\nummac$ may grow with the total data size $\totalnumobs$,
%but we must ensure that the regularization parameter $\lambda$ is
%large enough to ensure that $\nummac \ge \totalnumobs /
%(\momentbound^2 \log(d) \fulltracelambda)^2$.

Finally, it is worth noting that in practice, the optimal choice for
the regularization $\lambda$ may not be known \emph{a priori}. Thus it
seems that an adaptive choice of the regularization $\lambda$ would be
desirable (see, for example, \citet[Chapter
  3]{Tsybakov2009}). Cross-validation or other types of unbiased risk
estimation are natural choices, however, it is at this point unclear
how to achieve such behavior in the distributed setting we study, that
is, where estimates $\fapprox_i$ depend only on the $i$th local
dataset.  We leave this as an open question.

%%%%%%%%%%%%%%%%%%%%%%%%%%%%%%%%%%%%%%%%%%%%%%%%%%%%%%%%%%%%%%%%%%%%%%%%%%

\subsection{Some consequences}
\label{sec:corollaries}

We now turn to deriving some explicit consequences of our main
theorems for specific classes of reproducing kernel Hilbert spaces.
In each case, our derivation follows the broad outline given the the
remarks following Theorem~\ref{theorem:error-bound}: we first choose
the regularization parameter $\lambda$ to balance the bias and
variance terms, and then show, by comparison to known minimax lower
bounds, that the resulting upper bound is optimal.  Finally, we derive
an upper bound on the number of subsampled data sets $\nummac$ for
which the minimax optimal convergence rate can still be achieved.

\subsubsection{Finite-rank kernels}

Our first corollary applies to problems for which the kernel has
finite rank $r$, meaning that its eigenvalues satisfy $\eigenvalue_j =
0$ for all $j > r$.  Examples of such finite rank kernels include the
linear kernel $K(x, x') = \inprod{x}{x'}_{\real^d}$, which has rank at
most $r = d$; and the kernel $K(x, x) = (1 + x \, x')^m$ generating
polynomials of degree $m$, which has rank at most $r = m + 1$.
\begin{corollary}
  \label{corollary:low-rank-kernel}
  For a kernel with rank $r$, consider the output of the
  \algname\ algorithm with \mbox{$\lambda = \kerrank / \totalobs$.}
  Suppose that Assumption~\ref{assu:function-space} and
  Assumptions~\ref{assumption:kernel}
  (or~\ref{assumption:bounded-kernel}) hold, and that the number of
  processors $\nummac$ satisfy the bound
  \begin{equation*}
    \nummac \le c \frac{\totalobs^{\frac{\moment - 4}{\moment - 2}}}{\kerrank^2
      \momentbound^{\frac{4 \moment}{\moment - 2}}
      \log^{\frac{\moment}{\moment - 2}} \kerrank} ~~
    \mbox{(Assumption~\ref{assumption:kernel})} ~~~~ \mbox{or} ~~~~ \nummac
    \le c \frac{\totalobs}{\kerrank^2 \momentbound^4 \log \totalobs} ~~
    \mbox{(Assumption~\ref{assumption:bounded-kernel})},
  \end{equation*}
  where $c$ is a universal (numerical) constant. For
  suitably large $\totalobs$, the mean-squared error
  is bounded as
  \begin{align}
    \label{EqnFiniteRankRate}
    \E \left[ \ltwo{\funcavg - \fopt}^2\right] & = \order(1)
    \frac{\stddev^2 \kerrank}{\totalobs}.
  \end{align}
\end{corollary}

For finite-rank kernels, the rate~\eqref{EqnFiniteRankRate} is known
to be minimax-optimal, meaning that there is a universal constant $c'
> 0$ such that
\begin{align}
  \inf_{\ftil} \sup_{\fstar \in \Ball_\HilbertSpace(1)} \Exs[\|\ftil -
    \fstar\|_2^2] & \geq c' \, \frac{\kerrank}{\totalobs},
\end{align}
where the infimum ranges over all estimators $\ftil$ based on
observing all $\totalobs$ samples (and with no constraints on memory
and/or computation).  This lower bound follows from Theorem 2(a) of
Raskutti et al.~\cite{RasWaiYu12} with $s = d = 1$.

\subsubsection{Polynomially decaying eigenvalues}

Our next corollary applies to kernel operators with eigenvalues
that obey a bound of the form
\begin{align}
  \label{EqnPolyDecay}
  \eigenvalue_j & \leq C \, j^{-2 \smoothness} \quad \mbox{for all $j =
    1, 2, \ldots$,}
\end{align}
where $C$ is a universal constant, and $\smoothness > 1/2$
parameterizes the decay rate.  Kernels with polynomial decaying
eigenvalues include those that underlie for the Sobolev spaces with
different smoothness orders~\cite[e.g.][]{Birman1967,Gu2002}.  As a
concrete example, the first-order Sobolev kernel $K(x, x') = 1 + \min
\{x, x'\}$ generates an RKHS of Lipschitz functions with smoothness
$\smoothness = 1$.  Other higher-order Sobolev kernels also exhibit
polynomial eigendecay with larger values of the parameter
$\smoothness$.

\begin{corollary}
  \label{corollary:sobolev-kernel}
  For any kernel with $\smoothness$-polynomial
  eigendecay~\eqref{EqnPolyDecay}, consider the output of the
  \algname\ algorithm with $\lambda = (1/\totalobs)^{\frac{2
      \smoothness}{2 \smoothness+1}}$.  Suppose that
  Assumption~\ref{assu:function-space} and
  Assumption~\ref{assumption:kernel}
  (or~\ref{assumption:bounded-kernel}) hold, and that the number of
  processors satisfy the bound
  \begin{equation*}
    \nummac \le c \Bigg(\frac{\totalobs^{\frac{2(\moment - 4)\smoothness
          - \moment}{(2 \smoothness + 1)}}}{ \momentbound^{4 \moment}
      \log^\moment \totalobs} \Bigg)^{\frac{1}{\moment - 2}} ~~
    \mbox{(Assumption~\ref{assumption:kernel})} ~~~~ \mbox{or} ~~~~
    \nummac \le c \frac{\totalobs^{\frac{2 \smoothness - 1}{2 \smoothness
          + 1}}}{\momentbound^4 \log \totalobs} ~~
    \mbox{(Assumption~\ref{assumption:bounded-kernel})},
  \end{equation*}
  where $c$ is a constant only depending on $\smoothness$.  Then the
  mean-squared error is bounded as
  \begin{align}
    \label{EqnSobolevBound}
    \E \left[ \ltwo{\funcavg - \fopt}^2 \right] & = \order \left( \Big
    (\frac{\sigma^2}{\totalobs} \Big)^{\frac{2 \smoothness}{2
        \smoothness + 1}} \right).
  \end{align}
\end{corollary}

The upper bound~\eqref{EqnSobolevBound} is unimprovable up to constant
factors, as shown by known minimax bounds on estimation error in
Sobolev spaces~\cite{Stone82,Tsybakov2009}; see also Theorem 2(b) of
Raskutti et al.~\cite{RasWaiYu12}.

%%%%%%%%%%%%%%%%%%%%%%%%%%%%%%%%%%%%%%%%%%%%%%%%%%%%%%%%%%%%%%%%%%%%%%%%%%%

\subsubsection{Exponentially decaying eigenvalues}

Our final corollary applies to kernel operators with eigenvalues that
obey a bound of the form
\begin{align}
  \label{EqnExpDecay}
  \eigenvalue_j & \leq c_1 \, \exp(-c_2 j^2)  \quad \mbox{for all $j =
    1, 2, \ldots$,}
\end{align}
for strictly positive constants $(c_1, c_2)$.  Such classes include
the RKHS generated by the Gaussian kernel $K(x, x') = \exp(-\!\ltwo{x
  - x'}^2)$.
\begin{corollary}
  \label{corollary:gaussian-kernel}
  For a kernel with exponential eigendecay~\eqref{EqnExpDecay}, consider
  the output of the \algname\ algorithm with $\lambda = 1/\totalobs$.
  Suppose that Assumption~\ref{assu:function-space} and
  Assumption~\ref{assumption:kernel}
  (or~\ref{assumption:bounded-kernel}) hold, and that the number of
  processors satisfy the bound
  \begin{equation*}
    \nummac \le c \frac{\totalobs^{\frac{\moment - 4}{\moment - 2}}}{
      \momentbound^{\frac{4 \moment}{\moment - 2}} \log^{\frac{2
          \moment - 1}{\moment - 2}} \totalobs} ~~
    \mbox{(Assumption~\ref{assumption:kernel})} ~~~~ \mbox{or} ~~~~ \nummac
    \le c \frac{\totalobs}{\momentbound^4 \log^2 \totalobs} ~~
    \mbox{(Assumption~\ref{assumption:bounded-kernel})},
  \end{equation*}
  where $c$ is a constant only depending on $c_2$.  Then the mean-squared
  error is bounded as
  \begin{align}
    \label{EqnExpBound}
    \E\left [\ltwo {\funcavg - \fopt}^2\right] & = \order
    \left(\sigma^2 \, \frac{\sqrt{\log \totalobs}}{\totalobs} \right).
  \end{align}
\end{corollary}
\noindent The upper bound~\eqref{EqnExpBound} is also minimax optimal
for the exponential kernel classes, which behave like a finite-rank
kernel with effective rank $\sqrt{\log \totalobs}$.

\paragraph{Summary:} Each corollary gives a critical threshold for the
number $\nummac$ of data partitions: as long as $\nummac$ is below
this threshold, we see that the decomposition-based
\algname\ algorithm gives the optimal rate of convergence.  It is
interesting to note that the number of splits may be quite large: each
grows asymptotically with $\totalobs$ whenever the basis functions
have more than four moments (viz.
Assumption~\ref{assumption:kernel}).  Moreover, the \algname\ method
can attain these optimal convergence rates while using substantially
less computation than standard kernel ridge regression methods.

% Local Variables:
% TeX-master: "kernel-regression"
% End:

%% file: major-proofs.tex
\section{Proofs of Theorem~\ref{theorem:error-bound} and related results}
\label{sec:proof-results-and-corollaries}

We now turn to the proof of Theorem~\ref{theorem:error-bound} and
Corollaries~\ref{corollary:low-rank-kernel}
through~~\ref{corollary:gaussian-kernel}.  This section contains only a
high-level view of proof of Theorem~\ref{theorem:error-bound}; we defer more
technical aspects to the appendices.

%%%%%%%%%%%%%%%%%%%%%%%%%%%%%%%%%%%%%%%%%%%%%%%%%%%%%%%%%%%%%%%%%%%%%%%%

\subsection{Proof of Theorem~\ref{theorem:error-bound}}
\label{sec:proof-error-bound}

Using the definition of the averaged estimate $\funcavg =
\frac{1}{\nummac} \sum_{i=1}^\nummac \fapprox_i$, a bit of algebra
yields
\begin{align*}
  \E[\ltwo{ \funcavg - \functrue}^2] & = \E[\ltwo{(\funcavg -
      \E[\funcavg]) + (\E[\funcavg] - \functrue)}^2] \\ & =
  \E[\ltwo{\funcavg - \E[\funcavg]}^2] + \ltwo{\E[\funcavg] -
    \functrue }^2 + 2\E[\langle \funcavg - \E[\funcavg], \E[\funcavg]
    - \functrue\rangle_\lsspace] \\ & =
  \E\bigg[\ltwobigg{\frac{1}{\nmachine}\sum_{i=1}^\nmachine
      (\funcapprox_i - \E[\funcapprox_i])}^2\bigg] +
  \ltwo{\E[\funcavg] - \functrue }^2,
\end{align*}
where we used the fact that $\E[\fapprox_i] = \E[\funcavg]$ for each
$i \in [\nummac]$.  Using this unbiasedness once more, we bound the
variance of the terms $\fapprox_i - \E[\funcavg]$ to see that
\begin{align}
  \E\left[\ltwo{\funcavg - \functrue}^2\right] & \le
  \frac{1}{\nmachine} \E\left[\ltwos{\funcapprox_1 -
      \E[\fapprox_1]}^2\right] + \ltwos{\E[\funcapprox_1] -
    \functrue}^2 \nonumber \\ & \le \frac{1}{\nmachine}
  \E\left[\ltwos{\funcapprox_1 - \fopt}^2\right] +
  \ltwos{\E[\funcapprox_1] - \functrue}^2,
  \label{eqn:error-inequality}
\end{align}
where we have used the fact that $\E[\fapprox_i]$ minimizes
$\E[\ltwos{\fapprox_i - f}^2]$ over $f \in \hilbertspace$.

The error
bound~\eqref{eqn:error-inequality} suggests our strategy: we upper
bound $\E[\|\funcapprox_1 - \functrue\|_2^2]$ and $\|
\E[\funcapprox_1] - \functrue \|_2^2$ respectively.
Based on
equation~\eqref{eqn:sub-krr}, the estimate $\funcapprox_1$ is obtained
from a standard kernel ridge regression with sample size $\numobs =
\overallsize/\nmachine$ and ridge parameter $\lambda$.  Accordingly,
the following two auxiliary results provide bounds on these two terms,
where the reader should recall the definitions of $\maxlog$ and
$\tailsum$ from equation~\eqref{eqn:define-shorthands}. In each lemma,
$C$ represents a universal (numerical) constant.

%\mjwcomment{Where was $\tracelambda$ defined?}

\begin{lemma}[Bias bound]
  \label{lemma:norm-expectation-moments}
  Under Assumptions~\ref{assumption:kernel}
  and~\ref{assu:function-space}, for each $d = 1, 2, \ldots$, we have
  \begin{align}
    \ltwos{\E[\fapprox] - \fopt}^2 & \le 8 \lambda \hnorm{\fopt}^2 +
    \frac{8\momentbound^4 \hnorm{\fopt}^2 \tr(K) \tailsum}{\lambda} +
    \left( C \maxlog \frac{\momentbound^2 \fulltracelambda}{\sqrt{\numobs}}
    \right)^k \ltwo{\fopt}^2.
  \end{align}
\end{lemma}

\begin{lemma}[Variance bound]
  \label{lemma:variance-bound}
  Under Assumptions~\ref{assumption:kernel}
  and~\ref{assu:function-space}, for each $d = 1, 2, \ldots$, we have
  \begin{multline}
    \E[\ltwos{\fapprox - \fopt}^2] \le 12 \lambda \hnorm{\fopt}^2 +
    \frac{12 \stddev^2\fulltracelambda}{\numobs} \\
    + \left(\frac{2
      \stddev^2}{\lambda} + 4 \hnorm{\fopt}^2\right)
    \left(\eigenvalue_{d+1} + \frac{12 \momentbound^4 \tr(K)
      \tailsum}{\lambda} + \left(C \maxlog \frac{\momentbound^2
      \fulltracelambda}{\sqrt{\numobs}}\right)^k \ltwo{\fopt}^2\right).
  \end{multline}
\end{lemma}
\noindent The proofs of these lemmas, contained in
Appendices~\ref{sec:proof-bias} and~\ref{sec:proof-variance}
respectively, constitute one main technical contribution of this
paper. \\

Given these two lemmas, the remainder of the theorem proof is
straightforward.  Combining the
inequality~\eqref{eqn:error-inequality} with
Lemmas~\ref{lemma:norm-expectation-moments}
and~\ref{lemma:variance-bound} yields the claim of
Theorem~\ref{theorem:error-bound}.

%%%%%%%%%%%%%%%%%%%%%%%%%%%%%%%%%%%%%%%%%%%%%%%%%%%%%%%%%%%%%%%%%%%%%%%

\subsection{Proof of Corollary~\ref{corollary:low-rank-kernel}}

We first present a general inequality bounding the size of $\nummac$
for which optimal convergence rates are possible. We assume that $d$
is chosen large enough that for some constant $c$, we have $c\log(2d)
\ge \moment$ in Theorem~\ref{theorem:error-bound}, and that the
regularization $\lambda$ has been chosen. In this case, inspection of
Theorem~\ref{theorem:error-bound} shows that if $\nummac$ is small
enough that
\begin{equation*}
  \left(\sqrt{\frac{\log d}{\totalobs / \nummac}} \momentbound^2
  \fulltracelambda
  \right)^\moment \frac{1}{\nummac \lambda} \le
  \frac{\fulltracelambda}{\totalobs},
\end{equation*}
then the term $\Term_3(d)$ provides a convergence rate given by
$\fulltracelambda / \totalobs$. Thus, solving the expression above for
$\nummac$, we find
\begin{equation*}
  \frac{\nummac \log d}{\totalobs}\momentbound^4 \fulltracelambda^2
  = \frac{\lambda^{2/\moment} \nummac^{2/\moment}
    \fulltracelambda^{2/\moment}}{\totalobs^{2/\moment}} ~~~
  \mbox{or} ~~~ \nummac^{\frac{\moment - 2}{\moment}} =
  \frac{\lambda^{\frac{2}{\moment}} \totalobs^{\frac{\moment -
        2}{\moment}}}{\fulltracelambda^{2\frac{\moment -
        1}{\moment}} \momentbound^4 \log d}.
\end{equation*}
Taking $(\moment - 2) / \moment$-th roots of both sides, we obtain
that if
\begin{equation}
  \label{eqn:general-num-machines}
  \nummac \le \frac{\lambda^{\frac{2}{\moment - 2}} \totalobs}{
    \fulltracelambda^{2\frac{\moment - 1}{\moment - 2}}
    \momentbound^{\frac{4\moment}{\moment - 2}}
    \log^{\frac{\moment}{\moment - 2}} d},
\end{equation}
then the term $\Term_3(d)$ of the bound~\eqref{EqnGoodUpper}
is $\order(\fulltracelambda/ \totalobs)$.

Now we apply the bound~\eqref{eqn:general-num-machines} in the case in the
corollary. Let us take $d = \kerrank$; then $\tailsum=\eigenvalue_{d + 1}=0$.
We find that $\fulltracelambda \le \kerrank$ since each of its terms is
bounded by 1, and we take $\lambda = \kerrank / \totalobs$. Evaluating the
expression~\eqref{eqn:general-num-machines} with this value, we arrive at
\begin{equation*}
  \nummac \le \frac{\totalobs^{\frac{\moment - 4}{\moment - 2}}}{
    \kerrank^2 \momentbound^{\frac{4 \moment}{\moment - 2}}
    \log^{\frac{\moment}{\moment - 2}} \kerrank}.
\end{equation*}
If we have sufficiently many moments that $\moment \ge \log
\totalobs$, and $\totalobs \ge \kerrank$ (for example, if the basis
functions $\basis_j$ have a uniform bound $\momentbound$), then we
may take $\moment = \log \totalobs$, which implies that
$\totalobs^{\frac{\moment - 4}{\moment - 2}} = \Omega(\totalobs)$,
and we replace $\log d = \log \kerrank$ with $\log \totalobs$ (we assume
$\totalobs \ge \kerrank$), by recalling
Theorem~\ref{theorem:error-bound}. Then so long as
\begin{equation*}
  \nummac \le c \frac{\totalobs}{\kerrank^2 \momentbound^4 \log \totalobs}
\end{equation*}
for some constant $c > 0$, we obtain an identical result.

%%%%%%%%%%%%%%%%%%%%%%%%%%%%%%%%%%%%%%%%%%%%%%%%%%%%%%%%%%%%%%%%%%%%%%%%%%

\subsection{Proof of Corollary~\ref{corollary:sobolev-kernel}}

We follow the program outlined in our remarks following
Theorem~\ref{theorem:error-bound}.  We must first choose $\lambda$
so that $\lambda = \fulltracelambda / \totalobs$. To that end, we
note that setting $\lambda = \totalobs^{-\frac{2 \smoothness}{2
    \smoothness + 1}}$ gives
\begin{align*}
  \fulltracelambda
  = \sum_{j = 1}^\infty \frac{1}{1
    + j^{2 \smoothness} \totalobs^{-\frac{2 \smoothness}{2 \smoothness + 1}}}
  & \le
  \totalobs^{\frac{1}{2 \smoothness + 1}}
  + \sum_{j > \totalobs^{\frac{1}{2 \smoothness + 1}}}
  \frac{1}{1 + j^{2\smoothness}
    \totalobs^{-\frac{2 \smoothness}{2 \smoothness + 1}}} \\
  & \le \totalobs^{\frac{1}{2 \smoothness + 1}}
  + \totalobs^{\frac{2 \smoothness}{2 \smoothness + 1}}
  \int_{\totalobs^{\frac{1}{2 \smoothness + 1}}}
  \frac{1}{u^{2 \smoothness}} du
  = \totalobs^{\frac{1}{2 \smoothness + 1}}
  + \frac{1}{2 \smoothness - 1}
  \totalobs^{\frac{1}{2 \smoothness + 1}}.
\end{align*}
Dividing by $\totalobs$, we find that $\lambda \approx \fulltracelambda /
\totalobs$, as desired.  Now we choose the truncation parameter $d$. By
choosing $d = \totalobs^t$ for some $t \in \R_+$, then we find that
$\eigenvalue_{d + 1} \lesssim \totalobs^{-2 \smoothness t}$ and an
integration yields $\tailsum \lesssim \totalobs^{-(2 \smoothness - 1) t}$.
Setting $t = 3 / (2 \smoothness - 1)$ guarantees that
$\eigenvalue_{d + 1} \lesssim \totalobs^{-3}$ and
$\tailsum \lesssim \totalobs^{-3}$; the corresponding terms in the
bound~\eqref{EqnGoodUpper} are thus negligible. Moreover,
we have for any finite $\moment$ that $\log d \gtrsim k$.

Applying the general bound~\eqref{eqn:general-num-machines} on
$\nummac$, we arrive at the inequality
\begin{equation*}
  \nummac \le c \frac{\totalobs^{-\frac{4 \smoothness}{
        (2 \smoothness + 1)(\moment - 2)}} \totalobs}{
    \totalobs^{\frac{2(\moment - 1)}{(2 \smoothness + 1)(\moment - 2)}}
    \momentbound^{\frac{4 \moment}{\moment - 2}} \log^{\frac{\moment}{
        \moment - 2}} \totalobs}
  = c \frac{\totalobs^{\frac{2(\moment - 4)\smoothness - \moment}{
        (2 \smoothness + 1)(\moment - 2)}}}{
    \momentbound^{\frac{4 \moment}{\moment - 2}}
    \log^{\frac{\moment}{\moment - 2}} \totalobs}.
\end{equation*}
Whenever this holds, we have convergence rate $\lambda =
\totalobs^{-\frac{2 \smoothness}{2 \smoothness + 1}}$. Now, let
Assumption~\ref{assumption:bounded-kernel} hold, and take $k = \log
\totalobs$. Then the above bound becomes (to a multiplicative constant
factor) $\totalobs^{\frac{2 \smoothness - 1}{2 \smoothness + 1}} /
\momentbound^4 \log \totalobs$, as claimed.

%%%%%%%%%%%%%%%%%%%%%%%%%%%%%%%%%%%%%%%%%%%%%%%%%%%%%%%%%%%%%%%%%%%%%%%%%%%%
\subsection{Proof of Corollary~\ref{corollary:gaussian-kernel}}

First, we set $\lambda = 1 / \totalobs$. Considering the sum
$\fulltracelambda = \sum_{j=1}^\infty \eigenvalue_j / (\eigenvalue_j +
\lambda)$, we see that for $j \le \sqrt{(\log \totalobs)/c_2}$, the
elements of the sum are bounded by $1$. For $j > \sqrt{(\log
  \totalobs)/c_2}$, we make the approximation
\begin{equation*}
  \sum_{j \ge \sqrt{(\log \totalobs)/c_2}}
  \frac{\eigenvalue_j}{\eigenvalue_j + \lambda} \le
  \frac{1}{\lambda} \sum_{j \ge \sqrt{(\log \totalobs)/c_2}} \eigenvalue_j
  \lesssim \totalobs \int_{\sqrt{(\log \totalobs)/c_2}}^\infty
  \exp(-c_2 t^2) dt = \order(1).
\end{equation*}
Thus we find that $\fulltracelambda + 1\le c \sqrt{\log \totalobs}$ for
some constant $c$. By choosing $d = \totalobs^2$, we have that the
tail sum and $(d + 1)$-th eigenvalue both satisfy $\eigenvalue_{d + 1}
\le \tailsum \lesssim c_2^{-1} \totalobs^{-4}$.  As a consequence, all
the terms involving $\tailsum$ or $\eigenvalue_{d + 1}$ in the
bound~\eqref{EqnGoodUpper} are negligible.

Recalling our inequality~\eqref{eqn:general-num-machines}, we thus
find that (under Assumption~\ref{assumption:kernel}), as long as the
number of partitions $\nummac$ satisfies
\begin{equation*}
  \nummac \le c\frac{\totalobs^{\frac{\moment - 4}{\moment - 2}}}{
    \momentbound^{\frac{4 \moment}{\moment - 2}}
    \log^{\frac{2\moment - 1}{\moment - 2}} \totalobs},
\end{equation*}
the convergence rate of $\funcavg$ to $\fopt$ is given by
$\fulltracelambda / \totalobs \simeq \sqrt{\log \totalobs} / \totalobs$.
Under the boundedness assumption~\ref{assumption:bounded-kernel}, as
we did in the proof of Corollary~\ref{corollary:low-rank-kernel}, we
take $\moment = \log \totalobs$ in
Theorem~\ref{theorem:error-bound}.  By inspection, this yields the
second statement of the corollary.

%%%%%%%%%%%%%%%%%%%%%%%%%%%%%%%%%%%%%%%%%%%%%%%%%%%%%%%%%%%%%%%%%%%%%%%%%%

% Local Variables:
% TeX-master: "kernel-regression"
% End:

%% file: oracle-proofs.tex
\section{Proof of Theorem~\ref{theorem:oracle} and related results}
\label{sec:proof-oracle}

In this section, we provide the proofs of
Theorem~\ref{theorem:oracle}, as well as the
bound~\eqref{EqnDefnResidualAlt} based on the alternative form of the
residual error.  As in the previous section, we present a high-level
proof, deferring more technical arguments to the appendices.

%%%%%%%%%%%%%%%%%%%%%%%%%%%%%%%%%%%%%%%%%%%%%%%%%%%%%%%%%%%%%%%%%%%%%%%%%%%
\subsection{Proof of Theorem~\ref{theorem:oracle}}

We begin by stating and proving two auxiliary claims:
\begin{subequations}
\begin{align}
\label{eqn:oracle-equivalence}
\E \left[(Y - f(X))^2\right] & = \E\left[(Y - \fopt(X))^2\right] +
\ltwo{f-\fopt}^2 ~~ \mbox{for~any~} f \in \lsspace, ~~~~\mbox{and} \\
\label{eqn:oracle-minimizer}
\foptimalreg & = \argmin_{\hnorm{f} \leq \radius} \ltwo{f-\fopt}^2.
\end{align}
\end{subequations}
Let us begin by proving equality~\eqref{eqn:oracle-equivalence}.
By adding and subtracting terms, we have
\begin{align*}
\E\left[(Y-\fopt(X))^2\right] &= \E\left[(Y-\fopt(X))^2\right] +
\ltwo{f-\fopt}^2 + 2\E[(f(X)-\fopt(X))\E[(Y-\fopt(X)) \mid X = x]] \\
& \stackrel{(i)}{=} \E\left[(Y-\fopt(X))^2\right] + \ltwo{f-\fopt}^2,
\end{align*}
where equality (i) follows since the random variable $Y - \fopt(X)$ is
mean-zero given $X = x$.

For the second equality~\eqref{eqn:oracle-minimizer}, consider any
function $f$ in the RKHS that satisfies $\hnorm{f} \leq \radius$.  The
definition of the minimizer $\foptimalreg$ guarantees that
\begin{align*}
\E \left[(\foptimalreg(X) - Y)^2\right] + \lambdabase \radius^2 & \leq
\E[(f(X) - Y)^2] + \lambdabase\hnorm{f}^2 \leq \E[(f(X) - Y)^2] +
\lambdabase \radius^2.
\end{align*}
This result combined with equation~\eqref{eqn:oracle-equivalence} establishes
the equality~\eqref{eqn:oracle-minimizer}. \\

We now turn to the proof of the theorem.  Applying H\"older's
inequality yields that
\begin{align}
  \E\left[\ltwo{\funcavg-\fopt}^2\right] & \leq \left(1 +
  \frac{1}{q}\right) \ltwo{\foptimalreg-\fopt}^2 + (1 + q)
  \ltwo{\funcavg - \foptimalreg}^2\nonumber\\ & = \left(1 +
  \frac{1}{q}\right) \inf_{\hnorm{f} \le \radius} \ltwo{f-\fopt}^2 +
  (1 + q) \ltwo{\funcavg - \foptimalreg}^2 \qquad \mbox{for all $q >
    0$,}
  \label{eqn:oracle-bound-separation}
\end{align}
where the second step follows from
equality~\eqref{eqn:oracle-minimizer}.  It thus suffices to upper
bound $\ltwo{\funcavg - \foptimalreg}^2$, and following the deduction
of inequality~\eqref{eqn:error-inequality}, we immediately obtain the
decomposition formula
\begin{align}
  \label{eqn:oracle-error-inequality}
  \E\left[\ltwo{\funcavg - \foptimalreg}^2\right] \leq
  \frac{1}{\nmachine} \E[\ltwos{\funcapprox_1 - \foptimalreg}^2] +
  \ltwos{ \E[\funcapprox_1] - \foptimalreg }^2,
\end{align}
where $\fapprox_1$ denotes the empirical minimizer for \emph{one}
of the subsampled datasets (i.e.\ the standard KRR solution on a sample
of size $\numobs = \totalnumobs / \nmachine$ with regularization
$\lambda$).
This suggests our strategy, which parallels our proof of
Theorem~\ref{theorem:error-bound}: we upper bound $\E[\ltwos{\funcapprox_1 -
  \foptimalreg}^2]$ and $\ltwos{\E[\funcapprox_1] - \foptimalreg}^2$,
respectively. In the rest of the proof, we let $\fapprox = \fapprox_1$
denote this solution.

Let the estimation error for a subsample be $\esterr = \funcapprox -
\foptimalreg$. Under Assumptions~\ref{assumption:kernel}
and~\ref{assumption:response-moment-bound}, we have the following two lemmas
bounding expression~\eqref{eqn:oracle-error-inequality}, which parallel
Lemmas~\ref{lemma:norm-expectation-moments} and~\ref{lemma:variance-bound} in
the case when $\fopt \in \HilbertSpace$. In each lemma, $C$ denotes a universal
constant.

\begin{lemma}\label{lemma:oracle-variance-bound}
  For all $d = 1, 2, \ldots$, we have
  \begin{multline}\label{eqn:oracle-variance-bound-expression}
    \E\left[\ltwo{\esterr}^2\right] \leq
    \frac{16(\lambdabase-\lambda)^2R^2}{\lambda} + \frac{8 \fulltracelambda
      \momentbound^2\altvarreg^2}{\numobs} \\ +
    \sqrt{32\radius^4+8\altvarreg^4/{\lambda^2}}
    \left(\eigenvalue_{d+1} +
    \frac{16\momentbound^4\tr(K)\tailsum}{\lambda} + \left(C \maxlog
    \frac{\momentbound^2 \fulltracelambda}{\sqrt{\numobs}}
    \right)^\moment\right).
  \end{multline}
\end{lemma}
Denoting the right hand side of inequality~\eqref{eqn:oracle-variance-bound-expression} 
by $D^2$, we have 

\begin{lemma}
  \label{lemma:oracle-norm-expectation-bound}
  For all $d = 1, 2, \ldots$, we have
  \begin{multline}
    \ltwo{\E[\esterr]}^2 \leq \frac{4 (\lambdabase-\lambda)^2
      \radius^2}{\lambda} + \frac{C\log^2(d) (\momentbound^2
      \fulltracelambda)^2}{\numobs} D^2 \\
    + \sqrt{32\radius^4 + 8 \altvarreg^4/{\lambda^2}}
    \left(\eigenvalue_{d+1} + \frac{4\momentbound^4\tr(K) \beta_d}{\lambda}
    \right).
  \end{multline}
\end{lemma}

\noindent See Appendices~\ref{appendix:proof-lemma-oracle-variance}
and~\ref{appendix:proof-lemma-oracle-expectation} for the proofs of these
two lemmas. \\

Given these two lemmas, we can now complete the proof of the theorem.
If the conditions~\eqref{eqn:oracle-condition} hold, we have
\begin{equation*}
\begin{split}
 & \tailsum \le \frac{\lambda}{(R^2 +
    \altvarreg^2/\lambda)\overallsize}, ~~~~ \eigenvalue_{d + 1} \le
  \frac{1}{(\radius^2 + \altvarreg^2 / \lambda) \totalnumobs}, \\
& \frac{\log^2(d)(\momentbound^2 \fulltracelambda)^2}{\numobs} \leq
  \frac{1}{\nmachine} ~~~~\mbox{and}~~~~ \left(\maxlog
  \frac{\momentbound^2 \fulltracelambda}{
    \sqrt{\numobs}}\right)^\moment \leq
  \frac{1}{(\radius^2 + \altvarreg^2/\lambda)\overallsize},
  \end{split}
\end{equation*}
so there is a universal constant $C'$ satisfying
\begin{equation*}
  \sqrt{32\radius^4+8\altvarreg^4/{\lambda^2}} \left(\eigenvalue_{d+1}
  + \frac{16\momentbound^4\tr(K)\tailsum}{\lambda} + \left(C \maxlog
  \frac{\momentbound^2 \fulltracelambda}{\sqrt{\numobs}}
  \right)^\moment \right) \le \frac{C'}{\totalnumobs}.
\end{equation*}
Consequently, Lemma~\ref{lemma:oracle-variance-bound} yields the upper
bound
\begin{align*}
\E[\ltwo{\esterr}^2] \le \frac{8 (\lambdabase - \lambda)^2
  \radius^2}{\lambda} + \frac{8 \fulltracelambda \momentbound^2
  \altvarreg^2}{\numobs} + \frac{C'}{\totalnumobs}.
\end{align*}
Since $\log^2(d)(\momentbound^2\fulltracelambda)^2 / \numobs \le
1/\nummac$ by assumption, we obtain
\begin{align*}
\E\left[\ltwos{\favg - \foptimalreg}^2\right] & \le \frac{C(\lambdabase
  - \lambda)^2 \radius^2}{\lambda \nummac} + \frac{C \fulltracelambda
  \momentbound^2 \altvarreg^2}{\totalnumobs} + \frac{C}{\totalnumobs
  \nummac} \\
& \quad ~ + \frac{4(\lambdabase - \lambda)^2 \radius^2}{\lambda} +
\frac{C (\lambdabase - \lambda)^2 \radius^2}{\lambda \nummac} + \frac{C
  \fulltracelambda \momentbound^2 \altvarreg^2}{\totalnumobs} +
\frac{C}{\totalnumobs \nummac} + \frac{C}{\totalnumobs},
\end{align*}
where $C$ is a universal constant (whose value is allowed to change
from line to line). Summing these bounds and using the condition that
$\lambda \ge \lambdabase$, we conclude that
\begin{equation*}
  \E\left[\ltwos{\favg - \foptimalreg}^2\right]
  \le \left(4 + \frac{C}{\nummac}\right)
  (\lambda - \lambdabase) \radius^2
  + \frac{C \fulltracelambda\momentbound^2  \altvarreg^2}{\totalnumobs}
  + \frac{C}{\totalnumobs}.
\end{equation*}
Combining this error bound with inequality~\eqref{eqn:oracle-bound-separation}
completes the proof.

%%%%%%%%%%%%%%%%%%%%%%%%%%%%%%%%%%%%%%%%%%%%%%%%%%%%%%%%%%%%%%%%%%%%%%%%
\subsection{Proof of the bound~\eqref{EqnDefnResidualAlt}}
\label{sec:proof-residual-alternate}

Using Theorem~\ref{theorem:oracle}, it suffices to show that
\begin{align}
\label{eqn:new-fourth-moment-bound-for-arbitrary-regularization}
 \altvarreg^4 & \leq 8 \tr(K)^2 \hnorms{\foptimalreg}^4 \momentbound^4
 + 8\ymoment^4.
\end{align}
By the tower property of expectations and Jensen's inequality, we have
\begin{align*}
\altvarreg^4 = \E [ (\E [ (\foptimalreg(x) - Y)^2 \mid X = x ] )^2 ] &
\leq \E[(\foptimalreg(X)-Y)^4] \; \leq \;8 \E[ ( \foptimalreg(X))^4] +
8 \E [Y^4 ].
\end{align*}
Since we have assumed that $\E[Y^4]\leq \ymoment^4$, the only
remaining step is to upper bound $\E[(\foptimalreg(X))^4]$. Let
$\foptimalreg$ have expansion $(\theta_1, \theta_2, \ldots)$ in the
basis $\{\basis_j\}$. For any $x\in\XC$, H\"older's inequality applied
with the conjugates $4/3$ and $4$ implies the upper bound
\begin{align}
  \foptimalreg(x)
  &= \sum_{j=1}^\infty (\eigenvalue_j^{1/4}\theta_j^{1/2})
  \frac{\theta_j^{1/2}\basis_j(x)}{\eigenvalue_j^{1/4}}
  \leq \left( \sum_{j=1}^\infty \eigenvalue_j^{1/3}\theta_j^{2/3}
  \right)^{3/4}
  \left( \sum_{j=1}^\infty \frac{\theta_j^2}{\eigenvalue_j}\basis_j^4(x)
  \right)^{1/4}.
  \label{eqn:new-fourth-moment-holder-decomposition}
\end{align}
Again applying H\"older's inequality---this time with
conjugates $3/2$ and $3$---to upper bound the first term in
the product in
inequality~\eqref{eqn:new-fourth-moment-holder-decomposition}, we obtain
\begin{align}
  \sum_{j=1}^\infty \eigenvalue_j^{1/3}\theta_j^{2/3}
  = \sum_{j=1}^\infty \eigenvalue_j^{2/3}
  \left(\frac{\theta_j^2}{\eigenvalue_j}\right)^{1/3}
  \leq \bigg( \sum_{j=1}^\infty  \eigenvalue_j \bigg)^{2/3}
  \Bigg( \sum_{j=1}^\infty\frac{\theta_j^2}{\eigenvalue_j} \Bigg)^{1/3}
  = \tr(K)^{2/3}
  \hnorms{\foptimalreg}^{2/3}.
  \label{eqn:new-fourth-moment-cofficient-upper-bound}
\end{align}
Combining inequalities~\eqref{eqn:new-fourth-moment-holder-decomposition}
and~\eqref{eqn:new-fourth-moment-cofficient-upper-bound}, we find that
\begin{align*}
  \E[(\foptimalreg(X))^4]
  \leq \tr(K)^2 \hnorms{\foptimalreg}^2
  \sum_{j=1}^\infty \frac{\theta_j^2}{\eigenvalue_j} \E[\basis_j^4(X)]
  \leq \tr(K)^2 \hnorms{\foptimalreg}^4 \momentbound^4,
\end{align*}
where we have used
Assumption~\ref{assumption:kernel}. This completes the proof
of
inequality~\eqref{eqn:new-fourth-moment-bound-for-arbitrary-regularization}.

% Local Variables:
% TeX-master: "kernel-regression"
% End:

%% file: simulations.tex
\section{Experimental results}

In this section, we report the results of experiments on both
simulated and real-world data designed to test the sharpness of our
theoretical predictions.

\subsection{Simulation studies}
\label{sec:experiments}

\newcommand{\normal}{\mathsf{N}} \newcommand{\uniform}{\mathsf{Uni}}

\begin{figure}
\begin{tabular}{ccc}
 \includegraphics[keepaspectratio,width=.45\columnwidth]{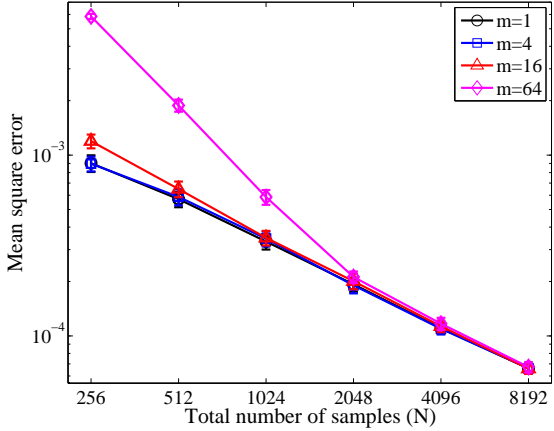}
 & ~~ &
 \includegraphics[keepaspectratio,
   width=.45\columnwidth]{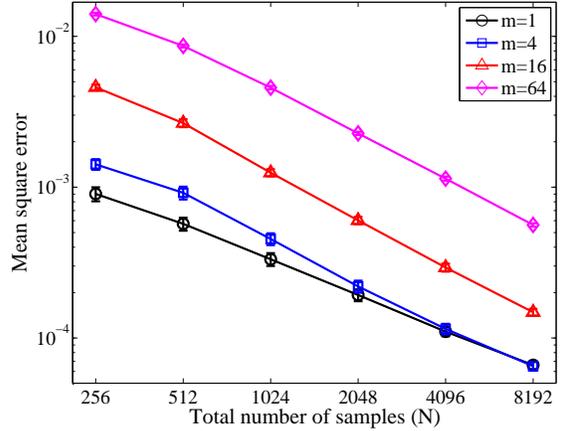} \\
(a) With under-regularization &~~&(b) Without under-regularization
  \end{tabular}
\label{fig:accuracy-compare}
  \caption{The squared $\lsspace$-norm between between the averaged
    estimate $\funcavg$ and the optimal solution $\foptimal$. (a)
    These plots correspond to the output of the \algname\ algorithm:
    each sub-problem is under-regularized by using $\lambda\sim
    \overallsize^{-2/3}$. (b) Analogous plots when each sub-problem is
    \emph{not} under-regularized---that is, with $\lambda \sim
    \nsample^{-2/3}$ is chosen as usual.}
\end{figure}

\begin{figure}
  \begin{tabular}{cc}
    \begin{minipage}{.6\columnwidth}
      \begin{center}
        \psfrag{log\(# of partitions\)/log\(# of samples\)}[Bc][Bc][.9]{
          $\log(\nummac) / \log(\totalnumobs)$}
        \includegraphics[width=.8\columnwidth]{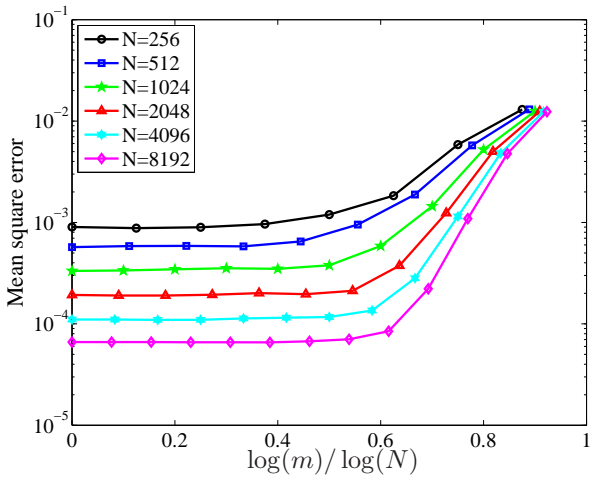}
      \end{center}
    \end{minipage} &
    \hspace{-1cm}
    \begin{minipage}{.42\columnwidth}
      \caption{\label{fig:threshold-compare} The mean-square error curves for
        fixed sample size but varied number of partitions. We are interested
        in the threshold of partitioning number $\nummac$ under which the
        optimal rate of convergence is achieved.}
    \end{minipage}
  \end{tabular}
\end{figure}

We begin by exploring the empirical performance of our
subsample-and-average methods for a non-parametric regression problem
on simulated datasets.  For all experiments in this section, we
simulate data from the regression model $y = \foptimal(x) + \noise$
for $x \in [0, 1]$, where $\foptimal(x) \defeq \min(x,1-x)$ is
$1$-Lipschitz, the noise variables $\noise \sim \normal(0, \stddev^2)$
are normally distributed with variance $\stddev^2 = 1/5$, and the
samples $x_i \sim \uniform[0, 1]$.  The Sobolev space of Lipschitz
functions on $[0, 1]$ has reproducing kernel $K(x, x') = 1 + \min\{x,
x'\}$ and norm $\hnorm{f}^2 = f^2(0) + \int_0^1 (f'(z))^2 dz$.  By
construction, the function $\foptimal(x) = \min(x, 1-x)$ satisfies
$\hnorm{\fopt} = 1$.  The kernel ridge regression estimator $\fapprox$
takes the form
\begin{equation}
  \label{eqn:krr-closed-form}
  \funcapprox = \sum_{i=1}^\overallsize \alpha_i \kernel(x_i, \cdot),
  ~~~ \mbox{where} ~~~
  \alpha = \left(K + \lambda \overallsize I
  \right)^{-1} y,
\end{equation}
and $K$ is the $\overallsize\times \overallsize$ Gram matrix and $I$ is the
$\overallsize\times \overallsize$ identity matrix.  Since the first-order
Sobolev kernel has eigenvalues~\cite{Gu2002} that scale as $\eigenvalue_j
\simeq (1/j)^2$, the minimax convergence rate in terms of squared
$L^2(\P)$-error is $\totalobs^{-2/3}$ (see e.g.~\cite{Tsybakov2009, Stone82,
  Caponnetto2007}).

By Corollary~\ref{corollary:sobolev-kernel} with $\smoothness = 1$,
this optimal rate of convergence can be achieved by \algname\ with
regularization parameter $\lambda \approx \totalobs^{-2/3}$, and as
long as the number of partitions $\nmachine$ satisfies
$\nmachine\lesssim \overallsize^{1/3}$. In each of our experiments, we
begin with a dataset of size $\totalnumobs = \nummac \nsample$, which
we partition uniformly at random into $\nummac$ disjoint subsets.  We
compute the local estimator $\fapprox_i$ for each of the $\nummac$
subsets using $\nsample$ samples via~\eqref{eqn:krr-closed-form},
where the Gram matrix is constructed using the $i$th batch of samples
(and $\nsample$ replaces $\totalobs$). We then compute $\funcavg = (1
/ \nummac) \sum_{i=1}^\nummac \fapprox_i$. Our experiments compare the
error of $\funcavg$ as a function of sample size $\totalobs$, the
number of partitions $\nummac$, and the regularization $\lambda$.

In Figure~\ref{fig:accuracy-compare}(a), we plot the error $\ltwos{\funcavg -
  \fopt}^2$ versus the total number of samples $\totalnumobs$, where
$\totalnumobs \in \{ 2^8,2^9,\dots,2^{13}\}$, using four different data
partitions $\nmachine\in\{1,4,16,64\}$. We execute each simulation $20$ times
to obtain standard errors for the plot.  The black circled curve
($\nmachine=1$) gives the baseline KRR error; if
the number of partitions \mbox{$\nummac \le 16$}, \algname\ has accuracy
comparable to the baseline algorithm. Even with $\nmachine=64$, \algname's
performance closely matches the full estimator for larger sample sizes
($\totalobs \ge 2^{11}$).  In the right plot
Figure~\ref{fig:accuracy-compare}(b), we perform an identical experiment, but
we over-regularize by choosing $\lambda = \nsample^{-2/3}$ rather than
$\lambda = \totalobs^{-2/3}$ in each of the $\nummac$ sub-problems, combining
the local estimates by averaging as usual. In contrast to
Figure~\ref{fig:accuracy-compare}(a), there is an obvious gap between the
performance of the algorithms when $\nummac = 1$ and $\nummac > 1$, as our
theory predicts.

It is also interesting to understand the number of partitions $\nummac$
into which a dataset of size $\overallsize$ may be divided while
maintaining good statistical performance.
According to Corollary~\ref{corollary:sobolev-kernel} with
$\smoothness = 1$, for the first-order Sobolev kernel, performance
degradation should be limited as long as $\nummac \lesssim
\overallsize^{1/3}$. In order to test this prediction,
Figure~\ref{fig:threshold-compare} plots the mean-square error
$\ltwos{\funcavg - \fopt}^2$ versus the ratio $\log(\nummac) /
\log(\totalnumobs)$.  Our theory predicts that even as the number of
partitions $\nummac$ may grow polynomially in $\totalnumobs$, the
error should grow only above some constant value of $\log(\nummac) /
\log(\totalnumobs)$.  As Figure~\ref{fig:threshold-compare} shows, the
point that $\ltwos{\funcavg - \fopt}$ begins to increase appears to be
around $\log(\nummac) \approx 0.45 \log(\totalnumobs)$ for reasonably
large $\overallsize$. This empirical performance is somewhat better
than the $(1/3)$ thresholded predicted by
Corollary~\ref{corollary:sobolev-kernel}, but it does confirm that the number
of partitions $\nummac$ can scale polynomially with
$\totalnumobs$ while retaining minimax optimality.

\begin{table}
  \begin{center}
  \begin{tabular}{|c|c|c|c|c|c|c|}
    \hline $\totalnumobs$ & & $\nummac = 1$ & $\nummac = 16$ &
    $\nummac = 64$ & $\nummac = 256$ & $\nummac = 1024$ \\ \hline
    \multirow{2}{*}{$ 2^{12}$} & Error & $1.26 \cdot
    10^{-4}$ & $1.33 \cdot 10^{-4}$ & $1.38 \cdot 10^{-4}$ &
    \multirow{2}{*}{N/A} & \multirow{2}{*}{N/A} \\ & Time & $1.12$
    ($0.03$) & $0.03$ ($0.01$) & $0.02$ ($0.00$) & & \\ \hline
    \multirow{2}{*}{$ 2^{13}$} & Error & $6.40 \cdot 10^{-5}$ & $6.29
    \cdot 10^{-5}$ & $6.72 \cdot 10^{-5}$ & \multirow{2}{*}{N/A} &
    \multirow{2}{*}{N/A} \\ & Time & $5.47$ ($0.22$) & $0.12$ ($0.03$)
    & $0.04$ ($0.00$) & & \\ \hline \multirow{2}{*}{$ 2^{14}$} & Error
    & $3.95 \cdot 10^{-5}$ & $4.06 \cdot 10^{-5}$ & $4.03 \cdot
    10^{-5}$ & $3.89 \cdot 10^{-5}$ & \multirow{2}{*}{N/A} \\ & Time &
    $30.16$ ($0.87$) & $0.59$ ($0.11$) & $0.11$ ($0.00$) & $0.06$
    ($0.00$) & \\ \hline \multirow{2}{*}{$ 2^{15}$} & Error &
    \multirow{2}{*}{Fail} & $2.90 \cdot 10^{-5}$ & $2.84 \cdot
    10^{-5}$ & $2.78 \cdot 10^{-5}$ & \multirow{2}{*}{N/A} \\ & Time &
    & $2.65$ ($0.04$) & $0.43$ ($0.02$) & $0.15$ ($0.01$) &\\ \hline
    \multirow{2}{*}{$ 2^{16}$} & Error & \multirow{2}{*}{Fail} & $1.75
    \cdot 10^{-5}$ & $1.73 \cdot 10^{-5}$ & $1.71 \cdot 10^{-5}$ &
    $1.67 \cdot 10^{-5}$ \\ & Time & & $16.65$ ($0.30$) & $2.21$
    ($0.06$) & $0.41$ ($0.01$) & $0.23$ ($0.01$) \\ \hline
    \multirow{2}{*}{$ 2^{17}$} & Error & \multirow{2}{*}{Fail} & $1.19
    \cdot 10^{-5}$ & $1.21 \cdot 10^{-5}$ & $1.25 \cdot 10^{-5}$ &
    $1.24 \cdot 10^{-5}$ \\ & Time & & $90.80$ ($3.71$) & $10.87$
    ($0.19$) & $1.88$ ($0.08$) & $0.60$ ($0.02$) \\ \hline
  \end{tabular}
  \end{center}
  \caption{\label{table:timing} Timing experiment giving
    $\ltwos{\funcavg - \fopt}^2$ as a function of number of partitions
    $\nummac$ and data size $\totalnumobs$, providing mean run-time (measured in second) for
    each number $\nummac$ of partitions and data size $\totalnumobs$.}
\end{table}

Our final experiment gives evidence for the improved time complexity
partitioning provides. Here we compare the amount of time required to
solve the KRR problem using the naive matrix
inversion~\eqref{eqn:krr-closed-form} for different partition sizes
$\nummac$ and provide the resulting squared errors $\ltwos{\funcavg -
  \fopt}^2$.  Although there are more sophisticated solution
strategies, we believe this is a reasonable proxy to exhibit \algname's
potential. In Table~\ref{table:timing}, we present the results of this
simulation, which we performed in Matlab using a Windows machine with
16GB of memory and a single-threaded 3.4Ghz processor. In each entry
of the table, we give the mean error of \algname\ and the mean amount
of time it took to run (with standard deviation over 10 simulations in
parentheses; the error rate standard deviations are an order of
magnitude smaller than the errors, so we do not report them). The
entries ``Fail'' correspond to out-of-memory failures because of the
large matrix inversion, while entries ``N/A'' indicate that
$\ltwos{\funcavg - \fopt}$ was significantly larger than the optimal
value (rendering time improvements meaningless). The table shows that
without sacrificing accuracy, decomposition via \algname\ can yield
substantial computational improvements.

% Local Variables:
% TeX-master: "kernel-regression"
% End:

%% file: year-prediction-experiment.tex
\subsection{Real data experiments}
\label{sec:real-data}

\begin{figure}[t]
  \begin{center}
  \includegraphics[width=.55\columnwidth]{%
    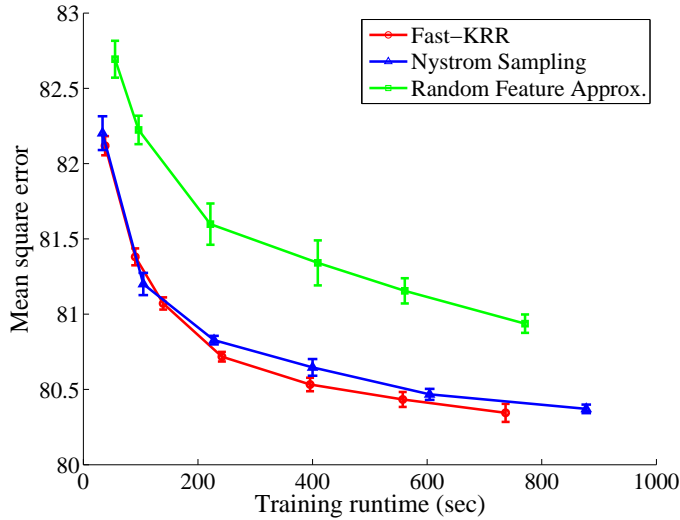}
    \caption{\label{fig:year-prediction-results} Results on year prediction
      on held-out test songs for \algname, Nystr\"om sampling, and
      random feature approximation. Error bars indicate standard deviations
      over ten experiments.}
  \end{center}
\end{figure}

We now turn to the results of experiments studying the performance of
\algname\ on the task of predicting the year in which a song was
released based on audio features associated with the song.  We use the
Million Song Dataset~\cite{Bertin2011}, which consists of 463,715
training examples and a second set of 51,630 testing examples. Each
example is a song (track) released between 1922 and 2011, and the song
is represented as a vector of timbre information computed about the
song.  Each sample consists of the pair $(x_i, y_i) \in \R^d \times
\R$, where $x_i \in \R^d$ is a $d = 90$-dimensional vector and $y_i
\in [1922, 2011]$ is the year in which the song was released.
(For further details, see the paper~\cite{Bertin2011}).

Our experiments with this dataset use the Gaussian radial
basis kernel
\begin{align}
  \label{eqn:year-prediction-gaussian-kernel}
  \kernel(x, x') = \exp\left(-\frac{\ltwo{x-x'}^2}{2\sigma^2}\right).
\end{align}
We normalize the feature vectors $x$ so that the timbre signals have
standard deviation $1$, and select the bandwidth parameter $\sigma =
6$ via cross-validation.  For regularization, we set
$\lambda = \totalobs^{-1}$; since the Gaussian kernel has
exponentially decaying eigenvalues (for typical distributions on $X$),
Corollary~\ref{corollary:gaussian-kernel} shows that this
regularization achieves the optimal rate of convergence for the
Hilbert space.

In Figure~\ref{fig:year-prediction-results}, we compare the
time-accuracy curve of \algname\ with two modern approximation-based
methods, plotting the mean-squared error between the predicted release
year and the actual year on test songs. The first baseline algorithm
is Nystr\"om subsampling~\cite{Williams2001}, where the kernel matrix
is approximated by a low-rank matrix of rank
$D\in\{1000,2000,3000,4000,5000,6000\}$. The second baseline approach
is an approximate form of kernel ridge regression using random
features~\cite{Rahimi2007}. The algorithm approximates the Gaussian
kernel~\eqref{eqn:year-prediction-gaussian-kernel} by the inner
product of two random feature vectors of dimensions
$D\in\{2000,3000,5000,7000,8500,10000\}$, and then solves the
resulting linear regression problem.  For the \algname\ algorithm, we
use seven partitions $\nummac\in\{32,38,48,64,96,128,256\}$ to test
the algorithm.  Each algorithm is executed 10 times to obtain standard
deviations (plotted as error-bars in
Figure~\ref{fig:year-prediction-results}).

As we see in Figure~\ref{fig:year-prediction-results}, for a fixed
time budget, \algname\ enjoys the best performance, though the margin
between \algname\ and Nystr\"om sampling is not substantial.  In spite
of this close performance between Nystr\"om sampling and the
divide-and-conquer \algname\ algorithm, it is worth noting that with
parallel computation, it is trivial to accelerate \algname\ $\nummac$
times; parallelizing approximation-based methods appears to be a
non-trivial task. Moreover, as our results in
Section~\ref{sec:main-result} indicate, \algname\ is minimax optimal
in many regimes. At the same time the conference version of this paper
was submitted, \citet{Bach2013} published the first results we know of
establishing convergence results in $\ell_2$-error for Nystr\"om
sampling; see the discussion for more detail. We note in passing that
standard linear regression with the original 90 features, while quite
fast with runtime on the order of 1 second (ignoring data loading),
has mean-squared-error $90.44$, which is significantly worse than the
kernel-based methods.

\begin{figure}
  \begin{center}
    \includegraphics[width=.55\columnwidth]{%
      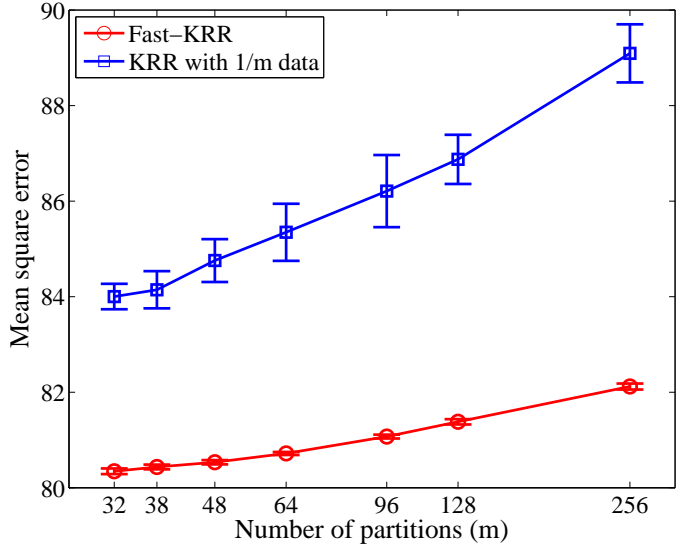}
    \caption{\label{fig:year-prediction-average-versus-fraction} Comparison of
      the performance of \algname\ to a standard KRR estimator using a
      fraction $1/m$ of the data.}
  \end{center}
\end{figure}

%\mjwcomment{Figure 4 looks a little bit messed-up (fonts bitmapped and
%  stretched).  Please regenerate a clean one.}

Our final experiment provides a sanity check: is the final averaging
step in \algname\ even necessary?  To this end, we compare
\algname\ with standard KRR using a fraction $1/\nummac$ of the
data. For the latter approach, we employ the standard regularization
$\lambda \approx (\totalobs/\nummac)^{-1}$. As
Figure~\ref{fig:year-prediction-average-versus-fraction} shows,
\algname\ achieves much lower error rates than KRR using only a
fraction of the data. Moreover, averaging stabilizes the estimators:
the standard deviations of the performance of \algname\ are negligible
compared to those for standard KRR.

% Local Variables:
% TeX-master: "kernel-regression"
% End:

%% file: discussion.tex
\section{Discussion}
\label{sec:discussion}
\newcommand{\margdim}{\widetilde{\gamma}(\lambda)}

In this paper, we present results establishing that our
decomposition-based algorithm for kernel ridge regression achieves
minimax optimal convergence rates whenever the number of splits
$\nummac$ of the data is not too large.  The error guarantees of our
method depend on the \emph{effective dimensionality} $\fulltracelambda
= \sum_{j=1}^\infty \eigenvalue_j / (\eigenvalue_j + \lambda)$ of the
kernel.  For any number of splits $\nummac \lesssim \totalnumobs /
\fulltracelambda^2$, our method achieves estimation error decreasing
as
\begin{equation*}
  \E\left[\ltwos{\bar{f} - \fopt}^2\right]
  \lesssim \lambda \hnorm{\fopt}^2
  + \frac{\stddev^2 \fulltracelambda}{\totalnumobs}.
\end{equation*}
(In particular, recall the bound~\eqref{eqn:competing-terms} following
Theorem~\ref{theorem:error-bound}).  Notably, this convergence rate is
minimax optimal, and we achieve substantial computational benefits
from the subsampling schemes, in that computational cost scales
(nearly) linearly in $\totalnumobs$.

\newcommand{\fulltravcelambdasq}{\gamma^2(\lambda)}

It is also interesting to consider the number of kernel evaluations
required to implement our method.  Our estimator requires $\nummac$
sub-matrices of the full kernel (Gram) matrix, each of size
$\totalnumobs / \nummac \times \totalnumobs / \nummac$.  Since the
method may use $\nummac \asymp \totalnumobs / \fulltracelambdasq$
machines, in the best case, it requires at most $\totalnumobs
\fulltracelambdasq$ kernel evaluations. By contrast, \citet{Bach2013}
shows that Nystr\"om-based subsampling can be used to form an
estimator within a constant factor of optimal as long as the number of
$\totalnumobs$-dimensional subsampled columns of the kernel matrix
scales roughly as the \emph{marginal dimension} $\margdim =
\totalnumobs \linf{\diag(K(K + \lambda \totalnumobs I)^{-1})}$.
Consequently, using roughly $\totalnumobs \margdim$ kernel
evaluations, Nystr\"om subsampling can achieve optimal convergence
rates. These two scalings--namely, $\totalnumobs \fulltracelambdasq$
versus $\totalnumobs \margdim$---are currently not comparable: in some
situations, such as when the data is not compactly supported,
$\margdim$ can scale linearly with $\totalnumobs$, while in others it
appears to scale roughly as the true effective dimensionality
$\fulltracelambda$.
A natural question arising from these lines of work is to
understand the true optimal scaling for these different estimators: is
one fundamentally better than the other? Are there natural computational
tradeoffs that can be leveraged at large scale?
As datasets grow substantially larger and more complex, these
questions should become even more important, and we hope to continue
to study them.

% Local Variables:
% TeX-master: "kernel-regression.tex"
% End:

%% file: acknowledgements.tex
\paragraph{Acknowledgements:}

We thank Francis Bach for interesting and enlightening conversations on the
connections between this work and his paper~\cite{Bach2013} and Yining Wang
for pointing out a mistake in an earlier version of this manuscript.  JCD was
supported by a National Defense Science and Engineering Graduate Fellowship
(NDSEG) and a Facebook PhD fellowship.  This work was partially supported by
ONR MURI grant N00014-11-1-0688 to MJW.

% Local Variables:
% TeX-master: "kernel-regression.tex"
% End:

%% file: bias-proof.tex
\section{Proof of Lemma~\ref{lemma:norm-expectation-moments}}
\label{sec:proof-bias}

\begin{table}
  \begin{center}
    \begin{tabular}{|p{.08\columnwidth} p{.85\columnwidth}|}
      \hline $\fapprox$ & Empirical KRR minimizer based on $\numobs$
      samples \\
      $\fopt$ & Optimal function generating data, where
      $y_i = \foptimal(x_i) + \noise_i$ \\
      $\esterr$ & Error
      $\funcapprox - \foptimal$ \\
      $\eval_x$ & RKHS evaluator $\eval_x
      \defeq K(x, \cdot)$, so $\<f, \eval_x\> = \<\eval_x, f\> = f(x)$  \\
      $\outprodmat$ & Operator mapping $\hilbertspace \rightarrow
      \hilbertspace$ defined as the outer product $\outprodmat \defeq
      \frac{1}{n} \sum_{i=1}^n \eval_{x_i} \otimes \eval_{x_i}$, so
      that $\outprodmat f = \frac{1}{n} \sum_{i=1}^n \<\eval_{x_i},
      f\> \eval_{x_i}$ \\
      $\basis_j$ & $j$th orthonormal basis vector
      for $\lspace$ \\
      $\coorderr_j$ & Basis coefficients of $\esterr$
      or $\E[\esterr \mid X]$ (depending on context), i.e.\ $\esterr =
      \sum_{j=1}^\infty \coorderr_j \basis_j$ \\
      $\theta_j$ & Basis
      coefficients of $\fopt$, i.e.\ $\fopt = \sum_{j = 1}^\infty
      \theta_j \basis_j$ \\
      $d$ & Integer-valued truncation point  \\
      $\eigmat$ & Diagonal matrix with $\eigmat =
      \diag(\eigenvalue_1, \ldots, \eigenvalue_d)$ \\
      $Q$ & Diagonal matrix with
      $Q = (I_{d \times d} + \lambda \eigmat^{-1})^\half$ \\
      $\basismat$ & $n
      \times d$ matrix with coordinates $\basismat_{ij} =
      \basis_j(x_i)$ \\
      $\truncate{v}$ & Truncation of vector $v$. For
      $v = \sum_j \nu_j \basis_j \in \hilbertspace$, defined as
      $\truncate{v} = \sum_{j=1}^d \nu_j \basis_j$; for $v \in
      \ell_2(\N)$ defined as $\truncate{v} = (v_1, \ldots, v_d)$
      \\
      $\detruncate{v}$ & Untruncated part of vector $v$, defined as
      $\detruncate{v} = (v_{d + 1}, v_{d + 1}, \ldots)$ \\
      $\tailsum$
      & The tail sum $\sum_{j > d} \eigenvalue_j$ \\
      $\fulltracelambda$ &
      The sum $\sum_{j = 1}^\infty 1 / (1 + \lambda / \eigenvalue_j)$
      \\
      $\maxlog$ & The maximum $\max\{\sqrt{\max\{\moment,
        \log(d)\}}, \max\{\moment, \log(d)\} / \numobs^{1/2 -
        1/\moment}\}$ \\
      % $\diffone$ & The constant $\sqrt[4]{32\hnorms{\foptimalreg}^4 +
      % 8\altvarreg^4/{\lambda^2}}$ \\
      \hline
    \end{tabular}
    \caption{\label{table:notation} Notation used in proofs}
  \end{center}
\end{table}

This appendix is devoted to the bias bound stated in
Lemma~\ref{lemma:norm-expectation-moments}. Let $X = \{x_i\}_{i=1}^n$
be shorthand for the design
matrix, and define the error vector $\esterr = \fapprox - \fopt$. By
Jensen's inequality, we have $\ltwo{\E[\esterr]} \le
\E[\ltwo{\E[\esterr \mid X]}]$, so it suffices to provide a bound
on $\ltwo{\E[\esterr \mid X]}$.  Throughout this proof and the
remainder of the paper, we represent the kernel evaluator by the
function $\eval_x$, where $\eval_x \defeq K(x, \cdot)$ and $f(x) =
\<\eval_x, f\>$ for any $f \in \hilbertspace$.  Using this notation,
the estimate $\fapprox$ minimizes the empirical objective
\begin{equation}
  \frac{1}{\numobs} \sum_{i=1}^\numobs \left(\<\eval_{x_i},
  f\>_\HilbertSpace - y_i\right)^2 + \lambda \hnorm{f}^2.
  \label{eqn:local-empirical-objective}
\end{equation}
This objective is Fr\'echet differentiable, and as a consequence, the
necessary and sufficient conditions for
optimality~\citep{Luenberger1969} of $\fapprox$ are that
\begin{equation}
  \label{eqn:gradient-optimality}
  \frac{1}{n} \sum_{i=1}^n (\langle \eval_{x_i}, \fapprox -
  \fopt\rangle_\HilbertSpace - \noise_i) + \lambda \fapprox =
  \frac{1}{\numobs} \sum_{i=1}^\numobs (\langle \eval_{x_i}, \fapprox
  \rangle_\HilbertSpace - y_i) + \lambda \fapprox = 0.
\end{equation}
Taking conditional expectations over the noise variables
$\{\noise_i\}_{i=1}^\numobs$ with the design $X = \{x_i\}_{i=1}^\numobs$
fixed, we find that
\begin{equation*}
  \frac{1}{\numobs} \sum_{i=1}^n \eval_{x_i} \<\eval_{x_i},
  \E[\esterr \mid X]\>
  + \lambda \E[\fapprox \mid X] = 0.
\end{equation*}
Define the sample covariance operator $\outprodmat \defeq
\frac{1}{\numobs} \sum_{i=1}^\numobs \eval_{x_i} \otimes \eval_{x_i}$.
Adding and subtracting $\lambda \fopt$ from the above equation yields
\begin{align}
\label{eqn:opt-gradient-equation}
(\outprodmat + \lambda I) \E[\esterr \mid X] & = - \lambda \fopt.
\end{align}
Consequently, we see we have $\hnorm{\E[\esterr \mid X]} \le
\hnorm{\fopt}$, since $\outprodmat \succeq 0$.

We now use a truncation argument to reduce the problem to a finite
dimensional problem.  To do so, we let $\coorderr \in
\ell_2(\N)$ denote the coefficients of $\E[\esterr
  \mid X]$ when expanded in the basis $\{\basis_j\}_{j=1}^\infty$:
\begin{align}
\E[\esterr \mid X] & = \sum_{j=1}^\infty \coorderr_j \basis_j, \quad
\mbox{with $\coorderr_j = \inprod{\Exs[\esterr \mid
      X]}{\basis_j}_{L^2(\P)}$.}
\end{align}
For a fixed $d \in \N$, define the vectors $\truncate{\coorderr}
\defeq (\coorderr_1, \ldots, \coorderr_d)$ and $\detruncate{\coorderr}
\defeq (\coorderr_{d + 1}, \coorderr_{d + 2}, \ldots)$ (we
suppress dependence on $d$ for convenience). By the
orthonormality of the collection $\{\basis_j\}$, we have
\begin{align}
\label{eqn:truncated-decomposition}
  \ltwo{\E[\esterr \mid X]}^2 & = \ltwo{\coorderr}^2 =
  \ltwos{\truncate{\coorderr}}^2 + \ltwos{\detruncate{\coorderr}}^2.
\end{align}
We control each of the elements of the
sum~\eqref{eqn:truncated-decomposition} in turn.

%%%%%%%%%%%%%%%%%%%%%%%%%%%%%%%%%%%%%%%%%%%%%%%%%%%%%%%%%%%%%%%%%%%%%%%%%%
\paragraph{Control of the term $\ltwos{\detruncate{\coorderr}}^2$:}

By definition, we have
\begin{align}
  \ltwos{\detruncate{\coorderr}}^2
  = \frac{\eigenvalue_{d +
      1}}{\eigenvalue_{d + 1}} \sum_{j = d + 1}^\infty \coorderr_j^2 &
  \, \leq \, \eigenvalue_{d + 1} \sum_{j = d+1}^\infty
  \frac{\coorderr_j^2}{\eigenvalue_j} % \nonumber \\
  \, \stackrel{(i)}{\leq} \, \eigenvalue_{d + 1} \hnorm{\E[\esterr \mid
      X]}^2 % \nonumber \\
  \label{eqn:kill-detruncated}
  \, \stackrel{(ii)} \, \leq \eigenvalue_{d + 1} \hnorm{\fopt}^2,
\end{align}
where inequality (i) follows since $\hnorm{\E[\esterr \mid X]}^2 =
\sum_{j=1}^\infty \frac{\coorderr_j^2}{\eigenvalue_j}$; and inequality
(ii) follows from the bound \mbox{$\hnorm{\E[\esterr \mid X]} \le
  \hnorm{\fopt}$,} which is a consequence of
equality~\eqref{eqn:opt-gradient-equation}.

%%%%%%%%%%%%%%%%%%%%%%%%%%%%%%%%%%%%%%%%%%%%%%%%%%%%%%%%%%%%%%%%%%%%%%%

\paragraph{Control of the term $\ltwos{\truncate{\coorderr}}^2$:}

Let $(\theta_1, \theta_2, \ldots)$ be the coefficients of $\fopt$ in
the basis $\{\basis_j\}$. In addition, define the matrices $\basismat
\in \R^{\numobs \times d}$ by
\begin{equation*}
  \basismat_{ij} = \basis_j(x_i) ~~ \mbox{for~} i \in \{1, \ldots,
  \numobs\}, \mbox{ and } j \in \{1, \ldots, d\}
\end{equation*}
and $\eigmat = \diag(\eigenvalue_1, \ldots, \eigenvalue_d) \succ 0 \in
\R^{d \times d}$.  Lastly, define the tail error vector $v \in
\R^\numobs$ by
\begin{align*}
v_i & \defn \sum_{j > d} \coorderr_j \basis_j(x_i) =
\E[\detruncate{\esterr}(x_i) \mid X].
\end{align*}
Let $l \in \N$ be arbitrary. Computing the (Hilbert) inner product of
the terms in equation~\eqref{eqn:opt-gradient-equation} with
$\basis_l$, we obtain
\begin{align*}
  \lefteqn{-\lambda \frac{\theta_l}{\eigenvalue_l}
    = \<\basis_l, -\lambda \fopt\>
    = \<\basis_l, (\outprodmat + \lambda) \E[\esterr \mid X]\>} \\
  & \quad = \frac{1}{\numobs} \sum_{i = 1}^\numobs \<\basis_l, \eval_{x_i}\>
  \<\eval_{x_i}, \E[\esterr \mid X]\> + \lambda \<\basis_l,
  \E[\esterr \mid X]\>
  = \frac{1}{\numobs}
  \sum_{i=1}^\numobs \basis_l(x_i) \E[\esterr(x_i) \mid X]
  + \lambda \frac{\coorderr_l}{\eigenvalue_l}.
\end{align*}
We can rewrite the final sum above using the fact that
$\esterr = \truncate{\esterr} + \detruncate{\esterr}$, which implies
\begin{equation*}
  \frac{1}{\nsample} \sum_{i=1}^\nsample
  \basis_l(x_i) \E[\esterr(x_i) \mid X]
  = \frac{1}{\nsample} \sum_{i=1}^\nsample \basis_l(x_i)
  \bigg(\sum_{j=1}^d
  \basis_j(x_i) \coorderr_j + \sum_{j > d} \basis_j(x_i) \coorderr_j\bigg)
\end{equation*}
Applying this equality for $l = 1, 2, \ldots, d$ yields
\begin{equation}
  \label{eqn:truncated-opt-gradient-equation}
  \left(\frac{1}{\numobs} \basismat^T \basismat + \lambda
  \eigmat^{-1}\right) \truncate{\coorderr} = - \lambda \eigmat^{-1}
  \truncate{\theta} - \frac{1}{n} \basismat^T v.
\end{equation}

We now show how the expression~\eqref{eqn:truncated-opt-gradient-equation}
gives us the desired bound in the lemma. By definining the shorthand matrix $Q
= (I + \lambda \eigmat^{-1})^{1/2}$, we have
\begin{equation*}
  \frac{1}{\numobs} \basismat^T \basismat + \lambda \eigmat^{-1}
  = I + \lambda \eigmat^{-1} + \frac{1}{n} \basismat^T \basismat - I
  = Q\left(I + Q^{-1}\left(\frac{1}{n} \basismat^T \basismat - I\right)Q^{-1}\right)Q.
\end{equation*}
As a consequence, we can rewrite
expression~\eqref{eqn:truncated-opt-gradient-equation} to
\begin{equation}
  \left(I + Q^{-1} \left(\frac{1}{n} \basismat^T \basismat - I\right) Q^{-1} \right) Q
  \truncate{\coorderr}
  = -\lambda Q^{-1} \eigmat^{-1} \truncate{\theta} - \frac{1}{n} Q^{-1} \basismat^T v.
  \label{eqn:inverted-truncated-opt-gradient}
\end{equation}
We now present a lemma bounding the terms in
equality~\eqref{eqn:inverted-truncated-opt-gradient} to control
$\truncate{\coorderr}$.

\begin{lemma}
  \label{LemBigLemma}
  The following bounds hold:
  \begin{subequations}
    \begin{align}
      \label{claim:truncated-fopt-small}
      \ltwo{\lambda Q^{-1} \eigmat^{-1} \truncate{\theta}}^2 \le \lambda
      \hnorm{\fopt}^2, \quad \mbox{and} \\
      \label{claim:detruncated-err-small}
      \E\left[\ltwo{\frac{1}{n} Q^{-1} \basismat^T v}^2\right] \le
      \frac{\momentbound^4 \hnorm{\fopt}^2 \tr(K) \tailsum}{\lambda}.
    \end{align}
  \end{subequations}
  Define the event $\event \defeq \left\{\matrixnorm{Q^{-1} \left(\frac{1}{n} \basismat^T \basismat - I\right) Q^{-1}} \leq 1/2\right\}$.  Under
  Assumption~\ref{assumption:kernel} with moment bound $\E[\basis_j(X)^{2
      \moment}] \le \momentbound^{2\moment}$, there exists a universal constant
  $C$ such that
  \begin{equation}
    \label{claim:bias-probability-small}
    \P(\event^c)
    \le \left(\max\left\{\sqrt{\moment \vee \log(d)}, \frac{\moment
      \vee \log(d)}{ \numobs^{1/2 - 1/\moment}} \right\} \frac{C
      \momentbound^2 \fulltracelambda}{\sqrt{\numobs}}\right)^\moment.
  \end{equation}
\end{lemma}
\noindent We defer the proof of this lemma to
Appendix~\ref{AppLemBigLemma}.

Based on this lemma, we can now complete the proof.  Whenever the
event $\event$ holds, we know that $I + Q^{-1}((1/\numobs)
\basismat^T \basismat - I)Q^{-1} \succeq (1/2) I$.  In particular, we have
\begin{equation*}
  \ltwos{\truncate{Q \coorderr}}^2
  \le 4 \ltwo{\lambda Q^{-1} \eigmat^{-1} \truncate{\theta}
    + (1/n) Q^{-1} \basismat^T v}^2
\end{equation*}
on $\event$, by Eq.~\eqref{eqn:inverted-truncated-opt-gradient}. Since $\ltwos{\truncate{Q \coorderr}}^2 \geq \ltwos{\truncate{\coorderr}}^2$,
the above inequality implies that
\begin{equation*}
  \ltwos{\truncate{\coorderr}}^2
  \le 4 \ltwo{\lambda Q^{-1} \eigmat^{-1} \truncate{\theta}
    + (1/n) Q^{-1} \basismat^T v}^2
\end{equation*}

Since $\event$ is $X$-measureable, we thus obtain
\begin{align*}
  \E\left[\ltwos{\truncate{\coorderr}}^2\right]
  & = \E\left[\indic{\event}\ltwos{\truncate{\coorderr}}^2\right]
  + \E\left[\indic{\event^c}\ltwos{\truncate{\coorderr}}^2\right] \\
  & \le 4 \E\left[\indic{\event}\ltwo{
      \lambda Q^{-1} \eigmat^{-1} \truncate{\theta}
      + (1/n) Q^{-1} \basismat^T v}^2\right]
  + \E\left[\indic{\event^c}\ltwos{\truncate{\coorderr}}^2\right].
\end{align*}
Applying the bounds~\eqref{claim:truncated-fopt-small}
and~\eqref{claim:detruncated-err-small}, along with the elementary
inequality $(a + b)^2 \le 2a^2 + 2b^2$, we have
\begin{equation}
  \label{eqn:almost-bias-lemma}
  \E\left[\ltwos{\truncate{\coorderr}}^2\right]
  \le 8 \lambda \hnorm{\fopt}^2
  + \frac{8 \momentbound^4 \hnorm{\fopt}^2 \tr(K) \tailsum}{\lambda}
  + \E\left[\indic{\event^c}\ltwos{\truncate{\coorderr}}^2\right].
\end{equation}
Now we use the fact that by the gradient optimality
condition~\eqref{eqn:opt-gradient-equation}, $\ltwo{\E[\esterr \mid
    X]}^2 \le \ltwo{\fopt}^2$.  Recalling the
shorthand~\eqref{eqn:define-maxlog} for $\maxlog$, we apply the
bound~\eqref{claim:bias-probability-small} to see
\begin{equation*}
  \E\left[\indic{\event^c}\ltwos{\truncate{\coorderr}}^2\right]
  \le \P(\event^c) \ltwo{f^*}^2
  \le \left(\frac{C \maxlog \momentbound^2 \fulltracelambda}{\sqrt{\numobs}}
  \right)^k \ltwo{\fopt}^2.
\end{equation*}
Combining this with the inequality~\eqref{eqn:almost-bias-lemma},
we obtain the desired statement of Lemma~\ref{lemma:norm-expectation-moments}.

%%%%%%%%%%%%%%%%%%%%%%%%%%%%%%%%%%%%%%%%%%%%%%%%%%%%%%%%%%%%%%%%%%%%%%%%%%%

\subsection{Proof of Lemma~\ref{LemBigLemma}}
\label{AppLemBigLemma}

\paragraph{Proof of bound~\eqref{claim:truncated-fopt-small}:}

Beginning with the proof of the
bound~\eqref{claim:truncated-fopt-small}, we have
\begin{align*}
\ltwo{Q^{-1}\eigmat^{-1} \truncate{\theta}}^2 & =
(\truncate{\theta})^T(\eigmat^2 + \lambda \eigmat)^{-1} \truncate{\theta}\\
& \leq (\truncate{\theta})^T (\lambda \eigmat)^{-1}\truncate{\theta}
= \frac{1}{\lambda} (\truncate{\theta})^T \eigmat^{-1}
\truncate{\theta} \leq \frac{1}{\lambda}\hnorm{\fopt}^2.
\end{align*}
Multiplying both sides by $\lambda^2$ gives the result.

%%%%%%%%%%%%%%%%%%%%%%%%%%%%%%%%%%%%%%%%%%%%%%%%%%%%%%%%%%%%%%%%%%%

\paragraph{Proof of bound~\eqref{claim:detruncated-err-small}:}
Next we turn to the proof of the
bound~\eqref{claim:detruncated-err-small}.  We begin by re-writing
$Q^{-1} \basismat^T v$ as the product of two components:
\begin{equation}
\label{eqn:magic-product}
\frac{1}{\nsample}Q^{-1}\basismat^T v = (\eigmat + \lambda I)^{-1/2}\left( \frac{1}{\nsample}\eigmat^{1/2}\basismat^T
v\right).
\end{equation}
The first matrix is a diagonal matrix whose operator norm is bounded:
\begin{equation}
\matrixnorm{(\eigmat + \lambda I)^{-1/2}} =
\max_{j\in [d]} \frac{1}{\sqrt{\eigenvalue_j + \lambda}} \le \frac{1}{
  \sqrt{\lambda}}.
  \label{eqn:operator-inv-eigs}
\end{equation}

For the second factor in the product~\eqref{eqn:magic-product}, the
analysis is a little more complicated. Let \mbox{$\basismat_\ell =
  (\basis_l(x_1), \ldots, \basis_l(x_n))$} be the $\ell$th column of
$\basismat$.  In this case,
\begin{align}
  \ltwo{\eigmat^{1/2} \basismat^T v}^2 & = \sum_{\ell=1}^d
  \eigenvalue_\ell(\basismat_\ell^T v)^2 \leq \sum_{\ell=1}^d
  \eigenvalue_\ell \ltwo{\basismat_\ell}^2 \ltwo{v}^2,
  \label{eqn:eigenmat-cauchy}
\end{align}
using the Cauchy-Schwarz inequality.  Taking expectations with respect
to the design $\Data$ and applying H\"older's inequality yields
\begin{equation*}
  \E[\ltwo{\basismat_\ell}^2 \ltwo{v}^2]
  \le \sqrt{\E[\ltwo{\basismat_\ell}^4]} \sqrt{\E[\ltwo{v}^4]}.
\end{equation*}
We bound each of the terms in this product in turn. For the first, we
have
\begin{align*}
  \E[\ltwo{\basismat_\ell}^4]
  & = \E\bigg[\bigg(\sum_{i=1}^\numobs \basis_\ell^2(X_i)\bigg)^2\bigg]
  = \E\bigg[\sum_{i, j = 1}^\numobs \basis_\ell^2(X_i) \basis_\ell^2(X_j) \bigg]
  \le \numobs^2 \E[\basis_\ell^4(X_1)]
  \le \numobs^2 \momentbound^4
\end{align*}
since the $X_i$ are i.i.d., $\E[\basis_\ell^2(X_1)] \le
\sqrt{\E[\basis_\ell^4(X_1)]}$, and $\E[\basis_\ell^4(X_1)] \le \momentbound^4$ by
assumption.  Turning to the term involving $v$, we have
\begin{equation*}
  v_i^2 = \bigg(\sum_{j > d} \coorderr_j \basis_j(x_i)\bigg)^2
  \le \bigg(\sum_{j > d} \frac{\coorderr_j^2}{\eigenvalue_j}\bigg)
  \bigg(\sum_{j > d} \eigenvalue_j \basis_j^2(x_i)\bigg)
\end{equation*}
by Cauchy-Schwarz. As a consequence, we find
\begin{align*}
  \E[\ltwo{v}^4]
  = \E\bigg[\bigg(\numobs \frac{1}{\numobs} \sum_{i=1}^n v_i^2 \bigg)^2\bigg]
  \le \numobs^2 \frac{1}{\numobs}
  \sum_{i=1}^\numobs \E[v_i^4]
  & \le \numobs
  \sum_{i=1}^\numobs \E\bigg[\bigg(\sum_{j > d}
    \frac{\coorderr_j^2}{\eigenvalue_j}\bigg)^2
    \bigg(\sum_{j > d} \eigenvalue_j \basis_j^2(X_i)\bigg)^2\bigg] \\
  & \le \numobs^2
    \E\bigg[\hnorm{\E[\esterr \mid X]}^4 \bigg(\sum_{j > d}
      \eigenvalue_j \basis_j^2(X_1)
    \bigg)^2\bigg],
\end{align*}
since the $X_i$ are i.i.d.  Using the fact that $\hnorm{\E[\esterr \mid X]}
\le \hnorm{\fopt}$, we expand the second square to find
\begin{equation*}
  \frac{1}{\numobs^2}\E[\ltwo{v}^4]
  \le \hnorm{\fopt}^4
  \sum_{j, k > d} \E\left[\eigenvalue_j \eigenvalue_k \basis_j^2(X_1)
    \basis_k^2(X_1)\right]
  \le \hnorm{\fopt}^4 \momentbound^4 \sum_{j, k > d} \eigenvalue_j \eigenvalue_k
  = \hnorm{\fopt}^4 \momentbound^4 \bigg(\sum_{j > d} \eigenvalue_j\bigg)^2.
\end{equation*}
Combining our bounds on $\ltwo{\basismat_\ell}$ and $\ltwo{v}$ with our initial
bound~\eqref{eqn:eigenmat-cauchy},
we obtain the inequality
\begin{equation*}
  \E\left[\ltwo{\eigmat^{1/2} \basismat^T v}^2\right]
  \le \sum_{l=1}^d \eigenvalue_\ell \sqrt{\numobs^2 \momentbound^4} \sqrt{
    \numobs^2 \hnorm{\fopt}^4
    \momentbound^4 \bigg(\sum_{j > d} \eigenvalue_j\bigg)^2}
  = \numobs^2\momentbound^4 \hnorm{\fopt}^2
  \bigg(\sum_{j > d} \eigenvalue_j\bigg) \sum_{l = 1}^d \eigenvalue_\ell.
\end{equation*}
Dividing by $\numobs^2$, recalling the definition of $\tailsum = \sum_{j >
  d} \eigenvalue_j$, and noting that $\tr(K) \ge \sum_{l = 1}^d \eigenvalue_\ell$ shows
that
\begin{equation*}
  \E\left[\ltwo{\frac{1}{\numobs}\eigmat^{1/2} \basismat^T v}^2\right]
  \le \momentbound^4 \hnorm{\fopt}^2 \tailsum \tr(K).
\end{equation*}
Combining this inequality with our expansion~\eqref{eqn:magic-product}
and the bound~\eqref{eqn:operator-inv-eigs} yields the
claim~\eqref{claim:detruncated-err-small}.

%%%%%%%%%%%%%%%%%%%%%%%%%%%%%%%%%%%%%%%%%%%%%%%%%%%%%%%%%%%%%%%%%%%%%%%%%%%%%

\paragraph{Proof of bound~\eqref{claim:bias-probability-small}:}

Let us consider the expectation of the norm of the matrix
\mbox{$Q^{-1}((1/n) \basismat^T \basismat - I)Q^{-1}$.}  For each $i \in
     [\numobs]$, let $\pi_i = (\basis_1(x_i), \ldots, \basis_d(x_i))
     \in \R^d$ denote the $i$th row of the matrix $\basismat \in \R^{n
       \times d}$. Then we know that
\begin{equation*}
  Q^{-1}\left(\frac{1}{\numobs} \basismat^T \basismat
  - I\right)Q^{-1}
  = \frac{1}{\numobs}  \sum_{i=1}^\numobs Q^{-1} (\pi_i \pi_i^T - I)Q^{-1}.
\end{equation*}
Define the sequence of matrices
\begin{equation*}
  A_i \defeq Q^{-1} (\pi_i \pi_i^T - I) Q^{-1}
\end{equation*}
Then the matrices $A_i = A_i^T \in \R^{d \times d}$.
%, and moreover $\lambda_{\max}(A_i) = \matrixnorm{Q^{-1} (\pi_i \pi_i^T - I) Q^{-1}}$, and similarly for their averages~\cite{Bhatia1997}
Note that $\E[A_i] = 0$ and let $\varepsilon_i$ be i.i.d.\ $\{-1,
1\}$-valued Rademacher random variables. Applying a standard
symmetrization argument~\citep{Ledoux1991}, we find that for any $k
\ge 1$, we have
\begin{equation}
  \label{eqn:symmetrization}
  \E\left[\matrixnorm{Q^{-1}\left(\frac{1}{n} \basismat^T \basismat -
      I\right)Q^{-1}}^k\right] = \E\left[\matrixnorm{\frac{1}{n}
      \sum_{i=1}^n A_i}^k\right] \le 2^k
  \E\left[\matrixnorm{\frac{1}{n} \sum_{i=1}^n \varepsilon_i A_i}^k
    \right].
\end{equation}

\begin{lemma}
  \label{LemMyMatrixMoment}
  The quantity $\E\left[\matrixnorm{\frac{1}{\numobs} \sum_{i=1}^\numobs
      \varepsilon_i A_i}^k\right]^{1/k}$ is upper bounded by
  \begin{align}
    \sqrt{e (k \vee 2 \log(d))} \frac{\momentbound^2 \sum_{j=1}^d
      \frac{1}{1 + \lambda / \eigenvalue_j}}{ \sqrt{\numobs}} +
    \frac{4 e (k \vee 2 \log(d))}{\numobs^{1 - 1/k}} \bigg(\sum_{j=1}^d
    \frac{\momentbound^2}{1 + \lambda / \eigenvalue_j}\bigg).
    \label{eqn:consequence-matrix-moments}
  \end{align}
\end{lemma}

We take this lemma as given for the moment, returning to prove it shortly.
Recall the definition of the constant $\fulltracelambda = \sum_{j=1}^\infty 1
/ (1 + \lambda / \eigenvalue_j) \geq \sum_{j=1}^d 1 / (1 + \lambda /
\eigenvalue_j)$. Then using our symmetrization
inequality~\eqref{eqn:symmetrization}, we have
\begin{align}
  \lefteqn{\E\bigg[\matrixnorm{Q^{-1}\left(\frac{1}{n} \basismat^T \basismat
        - I\right)Q^{-1}}^k\bigg]} \nonumber\\
  & \le 2^k \left(\sqrt{e(k \vee \log(d))}
  \frac{\momentbound^2 \fulltracelambda}{
    \sqrt{\numobs}} +
  \frac{4 e (k \vee 2 \log(d))}{\numobs^{1 - 1/k}}
  \momentbound^2 \fulltracelambda
  \right)^k \nonumber\\
  & \le
  \max\left\{\sqrt{k \vee \log(d)},
  \frac{k \vee \log(d)}{\numobs^{1/2 - 1/k}}\right\}^k
  \left(\frac{C \momentbound^2 \fulltracelambda}{\sqrt{\numobs}}\right)^k, 
  \label{eqn:matrix-moment-bound}
\end{align}
where $C$ is a numerical constant.
By definition of the event $\event$, we see by Markov's inequality
that for any $k \in \R, k \ge 1$,
\begin{equation*}
  \P(\event^c)
  \le \frac{\E\left[
      \matrixnorm{Q^{-1}\left(\frac{1}{\numobs} \basismat^T \basismat - I
        \right)}^k\right]}{2^{-k}}
  \le \max\left\{\sqrt{k \vee \log(d)},
  \frac{k \vee \log(d)}{\numobs^{1/2 - 1/k}}\right\}^k
  \left(\frac{2C \momentbound^2 \fulltracelambda)}{\sqrt{\numobs}}\right)^k.
\end{equation*}
This completes the proof of the
bound~\eqref{claim:bias-probability-small}. \\

\vspace*{.2cm}

It remains to prove Lemma~\ref{LemMyMatrixMoment}, for which we make
use of the following result, due to~\citet[Theorem A.1(2)]{Chen2012}.
\begin{lemma}
  \label{lemma:matrix-moment}
  Let $X_i \in \R^{d \times d}$ be independent symmetrically distributed
  Hermitian matrices. Then
  \begin{align}
    \label{EqnTropp}
    \E \bigg [ \matrixnormbigg{\sum_{i=1}^n X_i}^k\bigg]^{1/k} \le \sqrt{e
      (k \vee 2 \log d)} \matrixnormbigg{\sum_{i=1}^n \E[X_i^2]}^{1/2} +
    2e (k \vee 2\log d) \left(\E[\max_i \matrixnorm{X_i}^k]\right)^{1/k}.
  \end{align}
\end{lemma}
\noindent The proof of Lemma~\ref{LemMyMatrixMoment} is based on
applying this inequality with $X_i = \varepsilon_i A_i/\numobs$, and then
bounding the two terms on the right-hand side of
inequality~\eqref{EqnTropp}.

We begin with the first term. Note that for any symmetric matrix $Z$, we have
the matrix inequalities $0 \preceq \E[(Z - \E[Z])^2] = \E[Z^2] - \E[Z]^2
\preceq \E[Z^2]$, so
\begin{equation*}
  \E[A_i^2] = \E[Q^{-1}(\pi_i \pi_i^T - I) Q^{-2}(\pi_i \pi_i^T - I) Q^{-1}]
  \preceq \E[Q^{-1}\pi_i \pi_i^T Q^{-2} \pi_i\pi_i^T Q^{-1}].
\end{equation*}
Instead of computing these moments directly, we provide bounds on their norms.
Since $\pi_i \pi_i^T$ is rank one and $Q$ is diagonal, we have
\begin{equation*}
  \matrixnorm{Q^{-1}\pi_i \pi_i^TQ^{-1}}
  = \pi_i^T(I + \lambda \eigmat^{-1})^{-1} \pi_i
  = \sum_{j=1}^d \frac{\basis_j(x_i)^2}{1 + \lambda / \eigenvalue_j}.
\end{equation*}
We also note that, for any $k \in \R, k \ge 1$, convexity implies that
\begin{align*}
  \bigg(\sum_{j=1}^d \frac{\basis_j(x_i)^2}{1 + \lambda / \eigenvalue_j}
  \bigg)^k
  & = \left(\frac{\sum_{l=1}^d 1 / (1 + \lambda / \eigenvalue_\ell)}{
    \sum_{l=1}^d 1 / (1 + \lambda / \eigenvalue_\ell)}
  \sum_{j=1}^d \frac{\basis_j(x_i)^2}{1 + \lambda / \eigenvalue_j}
  \right)^k \\
  & \le \bigg(\sum_{l=1}^d \frac{1}{1 + \lambda / \eigenvalue_\ell}\bigg)^k
  \frac{1}{\sum_{l=1}^d 1 / (1 + \lambda / \eigenvalue_\ell)}
  \sum_{j=1}^d \frac{\basis_j(x_i)^{2k}}{1 + \lambda / \eigenvalue_j},
\end{align*}
so if $\E[\basis_j(X_i)^{2\moment}] \le \momentbound^{2\moment}$, we obtain
\begin{equation}
  \E\bigg[\bigg(\sum_{j=1}^d \frac{\basis_j(x_i)^2}{
      1 + \lambda / \eigenvalue_j}\bigg)^k\bigg]
  \le \bigg(\sum_{j=1}^d \frac{1}{1 + \lambda / \eigenvalue_j}\bigg)^k
  \momentbound^{2k}.
  \label{eqn:power-basis-moment-bound}
\end{equation}

The sub-multiplicativity of the matrix norm implies $\matrixnorm{(Q^{-1}\pi_i
  \pi_i^TQ^{-1})^2} \le \matrixnorm{Q^{-1}\pi_i \pi_i^TQ^{-1}}^2$, and
consequently we have
\begin{align*}
  \E\left[\matrixnorm{(Q^{-1}\pi_i \pi_i^TQ^{-1})^2}\right] & \leq
  \E\left[\left( \pi_i^T(I + \lambda \eigmat^{-1})^{-1}
    \pi_i\right)^2\right]
  %
%%   & = \sqrt{\E\bigg[\bigg(\sum_{j=1}^d
%%       \frac{\basis_j(x_i)^2}{1 + \lambda /
%%         \eigenvalue_j}\bigg)^4\bigg]}
  \le \momentbound^4
  \bigg(\sum_{j=1}^d \frac{1}{1 + \lambda / \eigenvalue_j} \bigg)^2,
\end{align*}
where the final step follows from
inequality~\eqref{eqn:power-basis-moment-bound}.  Applying the triangle
inequality to the first term on the right-hand side of
Lemma~\ref{lemma:matrix-moment}, we have thus obtained the first term on the
right-hand side of expression~\eqref{eqn:consequence-matrix-moments}. \\

We now turn to the second term in
expression~\eqref{eqn:consequence-matrix-moments}.  For real $k \ge 1$, we
have
\begin{equation*}
  \E[\max_i \matrixnorm{\varepsilon_i A_i / \numobs}^k]
  = \frac{1}{\numobs^k}\E[\max_i \matrixnorm{A_i}^k]
  \le \frac{1}{\numobs^k} \sum_{i=1}^n \E[\matrixnorm{A_i}^k]
\end{equation*}
Since norms are sub-additive, we find that
\begin{equation*}
  \matrixnorm{A_i}^k
  \le 2^{k-1}
  \bigg(\sum_{j=1}^d \frac{\basis_j(x_i)^2}{1 + \lambda / \eigenvalue_j}
  \bigg)^k
  + 2^{k-1} \matrixnorm{Q^{-2}}^k
  = 2^{k-1}
  \bigg(\sum_{j=1}^d \frac{\basis_j(x_i)^2}{1 + \lambda / \eigenvalue_j}
  \bigg)^k
  + 2^{k-1} \bigg(\frac{1}{1 + \lambda/\eigenvalue_1}\bigg)^k.
\end{equation*}
Since $\momentbound \ge 1$ (recall that the $\basis_j$ are an orthonormal
basis), we apply
inequality~\eqref{eqn:power-basis-moment-bound}, to find
that
\begin{equation*}
  \E[\max_i \matrixnorm{\varepsilon_i A_i / \numobs}^k]
  \le \frac{1}{n^{k - 1}}\bigg[
    2^{k-1} \bigg(\sum_{j=1}^d \frac{1}{1 + \lambda / \eigenvalue_j}\bigg)^k
    \momentbound^{2k}
    + 2^{k-1} \bigg(\frac{1}{1 + \lambda / \eigenvalue_1}\bigg)^k
    \momentbound^{2k}\bigg].
\end{equation*}
Taking $k$th roots yields the second term in the
expression~\eqref{eqn:consequence-matrix-moments}.

% Local Variables:
% TeX-master: "kernel-regression"
% End:

%% file: variance-proof.tex
\section{Proof of Lemma~\ref{lemma:variance-bound}}
\label{sec:proof-variance}

This proof follows an outline similar to
Lemma~\ref{lemma:norm-expectation-moments}.  We begin with a simple
bound on $\hnorm{\esterr}$:
\begin{lemma}
\label{lemma:diff-bound}
Under Assumption~\ref{assu:function-space}, we have
$\E[\hnorm{\esterr}^2 \mid X] \le 2 \stddev^2 / \lambda + 4
\hnorm{\fopt}^2$.
\end{lemma}
\begin{proof}
  We have
  \begin{align*}
    \lambda \; \E[ \, \hnorms{\fapprox}^2 \mid \Data] & \leq
    \E\left[\frac{1}{\nsample}\sum_{i=1}^\nsample( \funcapprox(x_i) -
      \foptimal(x_i) - \noise_i)^2 + \lambda\hnorms{\fapprox}^2 \mid \Data
      \right] \\
    & \stackrel{(i)}{\leq} \frac{1}{\nsample}\sum_{i=1}^\nsample
    \E[\noise_i^2 \mid x_i] + \lambda\hnorm{\foptimal}^2 \\
    & \stackrel{(ii)}{\leq} \sigma^2 + \lambda\hnorm{\fopt}^2,
  \end{align*}
  where inequality (i) follows since $\fapprox$ minimizes the objective
  function~\eqref{eqn:krr-estimate}; and inequality (ii) uses the fact
  that $\E[\noise_i^2 \mid x_i]\leq \sigma^2$.  Applying the triangle
  inequality to $\hnorm{\esterr}$ along with the elementary inequality
  $(a + b)^2 \le 2 a^2 + 2b^2$, we find that
  \begin{align*}
    \E[\hnorm{\esterr}^2 \mid \Data] & \leq 2\hnorm{\fopt}^2 +
    2\E[\hnorms{\fapprox}^2 \mid \Data] \; \leq \; \frac{2
      \stddev^2}{\lambda} + 4 \hnorm{\fopt}^2,
  \end{align*}
  which completes the proof.
\end{proof}

With Lemma~\ref{lemma:diff-bound} in place, we now proceed to the
proof of the theorem proper. Recall from
Lemma~\ref{lemma:norm-expectation-moments} the optimality condition
\begin{equation}
  \label{eqn:gradient-optimality-v2}
  \frac{1}{\numobs} \sum_{i=1}^\numobs \eval_{x_i}
  (\langle \eval_{x_i}, \fapprox - \fopt\rangle - \noise_i)
  + \lambda \fapprox = 0.
\end{equation}
Now, let $\coorderr \in \ell_2(\N)$ be the expansion of the error
$\esterr$ in the basis $\{\basis_j\}$, so that $\esterr =
\sum_{j=1}^\infty \coorderr_j \basis_j$, and (again, as in
Lemma~\ref{lemma:norm-expectation-moments}), we choose $d \in \N$ and
truncate $\esterr$ via
\begin{equation*}
  \truncate{\esterr} \defeq
  \sum_{j = 1}^d \coorderr_j \basis_j
  ~~ \mbox{and} ~~
  \detruncate{\esterr}
  \defeq \esterr - \truncate{\esterr}
  = \sum_{j > d} \coorderr_j \basis_j.
\end{equation*}
Let $\truncate{\coorderr} \in \R^d$ and $\detruncate{\coorderr}$ denote
the corresponding vectors for the above.
As a consequence of the orthonormality of the basis functions, we have
\begin{equation}
  \label{eqn:truncated-decomposition-v2}
  \E[\ltwo{\esterr}^2]
  = \E[\ltwos{\truncate{\esterr}}^2]
  + \E[\ltwos{\detruncate{\esterr}}^2]
  = \E[\ltwos{\truncate{\coorderr}}^2]
  + \E[\ltwos{\detruncate{\coorderr}}^2].
\end{equation}
We bound each of the terms~\eqref{eqn:truncated-decomposition-v2} in
turn.

By Lemma~\ref{lemma:diff-bound}, the second term is upper bounded as
\begin{equation}
  \label{eqn:detruncated-part-is-small-with-noise}
  \E[\ltwos{\detruncate{\esterr}}^2]
  = \sum_{j > d} \E[\coorderr_j^2]
  \le \sum_{j > d} \frac{\eigenvalue_{d + 1}}{\eigenvalue_j}
  \E[\coorderr_j^2]
  = \eigenvalue_{d + 1} \E[\hnorms{\detruncate{\esterr}}^2]
  \le \eigenvalue_{d + 1}\left(\frac{2 \stddev^2}{\lambda}
  + 4 \hnorm{\fopt}^2\right).
\end{equation}

The remainder of the proof is devoted the bounding the term
$\E[\ltwos{\truncate{\esterr}}^2]$ in the
decomposition~\eqref{eqn:truncated-decomposition-v2}.  By taking the
Hilbert inner product of $\basis_k$ with the optimality
condition~\eqref{eqn:gradient-optimality-v2}, we find as in our
derivation of the matrix
equation~\eqref{eqn:truncated-opt-gradient-equation} that for each $k
\in \{1, \ldots, d\}$
\begin{equation*}
  \frac{1}{\nsample} \sum_{i=1}^\nsample \sum_{j=1}^d
  \basis_k(x_i)\basis_j(x_i) \coorderr_j
  + \frac{1}{\nsample}\sum_{i=1}^\nsample
  \basis_k(x_i) (\detruncate{\esterr}(x_i)
  -\noise_i) + \lambda\frac{\coorderr_k}{\eigenvalue_k} = 0.
\end{equation*}
Given the expansion $\fopt = \sum_{j = 1}^\infty \theta_j \basis_j$,
define the tail error vector $v \in \R^\numobs$ by $v_i = \sum_{j > d}
\coorderr_j \basis_j(x_i)$, and recall the definition of the
eigenvalue matrix $\eigmat = \diag(\eigenvalue_1, \ldots,
\eigenvalue_d) \in \R^{d \times d}$.  Given the matrix $\basismat$
defined by its coordinates $\basismat_{ij} = \basis_j(x_i)$, we have
\begin{equation}
  \label{eqn:truncated-gradient-v2}
  \left(\frac{1}{\numobs} \basismat^T \basismat + \lambda \eigmat^{-1}
  \right) \truncate{\coorderr}
  = - \lambda \eigmat^{-1} \truncate{\theta} - \frac{1}{\numobs}
  \basismat^T v + \frac{1}{\numobs} \basismat^T \noise.
\end{equation}
As in the proof of Lemma~\ref{lemma:norm-expectation-moments}, we find
that
\begin{equation}
  \label{eqn:inverted-truncated-gradient-noise}
  \left(I + Q^{-1}\left(\frac{1}{\numobs} \basismat^T \basismat
  - I\right)Q^{-1}\right) Q\truncate{\coorderr}
  = -\lambda Q^{-1} \eigmat^{-1} \truncate{\theta}
  - \frac{1}{\numobs} Q^{-1} \basismat^T v
  + \frac{1}{\numobs} Q^{-1} \basismat^T \noise,
\end{equation}
where we recall that $Q = (I + \lambda \eigmat^{-1})^{1/2}$.

We now recall the bounds~\eqref{claim:truncated-fopt-small}
and~\eqref{claim:bias-probability-small} from Lemma~\ref{LemBigLemma},
as well as the previously defined event \mbox{$\event \defeq
  \{\matrixnorm{Q^{-1}\left(\frac{1}{\numobs} \basismat^T \basismat
  - I\right)Q^{-1}} \le
  1/2 \}$.} When $\event$ occurs, the
expression~\eqref{eqn:inverted-truncated-gradient-noise} implies the
inequality
\begin{equation*}
  \ltwos{\truncate{\esterr}}^2 \le \ltwos{Q \truncate{\coorderr}}^2 \leq 4 \ltwo{-\lambda
    Q^{-1}\eigmat^{-1}\truncate{\theta} - (1 / \nsample)
    Q^{-1}\basismat^T v + (1 / \nsample) Q^{-1}\basismat^T \noise}^2.
\end{equation*}
When $\event$ fails to hold, Lemma~\ref{lemma:diff-bound} may still be
applied since $\event$ is measureable with respect to $\Data$.  Doing
so yields
\begin{align}
\lefteqn{\E[\ltwos{\truncate{\esterr}}^2] = \E[\indic{\event}
    \ltwos{\truncate{\esterr}}^2] + \E[\indic{\event^c}
    \ltwos{\truncate{\esterr}}^2]} \nonumber \\
& \leq 4 \E\left[\ltwo{-\lambda Q^{-1}\eigmat^{-1}\truncate{\theta} -
    (1 / \nsample) Q^{-1}\basismat^T v + (1 / \nsample)
    Q^{-1}\basismat^T \noise}^2\right] + \E\left[\indic{\event^c}
  \E[\ltwos{\truncate{\esterr}}^2 \mid \Data] \right] \nonumber \\ &
\leq 4 \E\left[\ltwo{\lambda Q^{-1}\eigmat^{-1}\truncate{\theta} +
    \frac{1}{\nsample} Q^{-1}\basismat^T v - \frac{1}{\nsample}
    Q^{-1}\basismat^T \noise}^2\right] + \P(\event^c)
\left(\frac{2\stddev^2}{\lambda} + 4 \hnorm{\fopt}^2\right).
  \label{eqn:bounding-truncation-with-noise}
\end{align}
Since the bound~\eqref{claim:bias-probability-small} still holds, it
remains to provide a bound on the first term in the
expression~\eqref{eqn:bounding-truncation-with-noise}.

As in the proof of Lemma~\ref{lemma:norm-expectation-moments}, we have
$\ltwos{\lambda Q^{-1} \eigmat^{-1} \truncate{\theta}}^2 \le \lambda \hnorm{\fopt}^2$
via the bound~\eqref{claim:truncated-fopt-small}.
Turning to the second term inside the norm, we claim that, under the
conditions of Lemma~\ref{lemma:variance-bound}, the following bound
holds:
\begin{align}
\label{claim:detruncated-err-small-noise}
 \E \left [\ltwo{(1 / \numobs) Q^{-1} \basismat^T v}^2\right] & \leq
 \frac{\momentbound^4 \tr(K) \tailsum (2 \stddev^2 / \lambda + 4
   \hnorm{\fopt}^2)}{\lambda}.
\end{align}
This claim is an analogue of our earlier
bound~\eqref{claim:detruncated-err-small}, and we prove it shortly.
Lastly, we bound the norm of $Q^{-1} \basismat^T \noise /
\numobs$. Noting that the diagional entries of $Q^{-1}$ are $1 / \sqrt{1 +
\lambda / \eigenvalue_j}$, we have
\begin{equation*}
  \E\left[\ltwo{Q^{-1} \basismat^T \noise}^2\right]
  = \sum_{j = 1}^d \sum_{i = 1}^\numobs
  \frac{1}{1 + \lambda / \eigenvalue_j}
  \E[\basis_j^2 (X_i) \noise_i^2]
\end{equation*}
Since $\E[\basis_j^2(X_i) \noise_i^2]
= \E[\basis_j^2(X_i) \E[\noise_i^2 \mid X_i]] \le \stddev^2$ by assumption,
we have the inequality
\begin{equation*}
  \E\left[\ltwo{(1 / \numobs) Q^{-1} \basismat^T\noise}^2\right]
  \le \frac{\stddev^2}{\numobs} \sum_{j=1}^d
  \frac{1}{1 + \lambda / \eigenvalue_j}.
\end{equation*}
The last sum is bounded by $(\stddev^2 / \numobs)
\fulltracelambda$.  Applying the inequality $(a + b + c)^2 \le
3a^2 + 3b^2 + 3c^2$ to inequality~\eqref{eqn:bounding-truncation-with-noise},
we obtain
\begin{align*}
  \E\left[\ltwos{\truncate{\esterr}}^2\right]
  & \le 12 \lambda \hnorm{\fopt}^2
  + \frac{12 \stddev^2 \fulltracelambda}{\numobs}
  + \left(\frac{2 \stddev^2}{\lambda} + 4 \hnorm{\fopt}^2\right)
  \left(\frac{12 \momentbound^4 \tr(K) \tailsum}{\lambda}
  + \P(\event^c)\right).
\end{align*}
Applying the bound~\eqref{claim:bias-probability-small} to control
$\P(\event^c)$ and bounding $\E[\ltwos{\detruncate{\esterr}}^2]$ using
inequality~\eqref{eqn:detruncated-part-is-small-with-noise} completes
the proof of the lemma. \\

\vspace*{.1in}

%%%%%%%%%%%%%%%%%%%%%%%%%%%%%%%%%%%%%%%%%%%%%%%%%%%%%%%%%%%%%%%%%%%%%%%%%

It remains to prove
bound~\eqref{claim:detruncated-err-small-noise}. Recalling the
inequality~\eqref{eqn:operator-inv-eigs}, we see that
\begin{equation}
  \label{eqn:err-small-noise-expansion-1}
  \ltwo{(1 / \numobs) Q^{-1} \basismat^T v}^2 \le \matrixnorm{Q^{-1}
    \eigmat^{-1/2}}^2 \ltwo{(1/\numobs) \eigmat^{1/2} \basismat^T v}^2
  \le \frac{1}{\lambda} \ltwo{(1/\numobs) \eigmat^{1/2} \basismat^T
    v}^2.
\end{equation}
Let $\basismat_\ell$ denote the $\ell$th column of the matrix
$\basismat$.  Taking expectations yields
\begin{equation*}
  \E\left[\ltwo{\eigmat^{1/2} \basismat^T v}^2\right]
  = \sum_{l = 1}^d \eigenvalue_\ell \E[\<\basismat_\ell, v\>^2]
  \le \sum_{l = 1}^d \eigenvalue_\ell
  \E\left[\ltwo{\basismat_\ell}^2 \ltwo{v}^2\right]
  = \sum_{l = 1}^d \eigenvalue_\ell
  \E\left[\ltwo{\basismat_\ell}^2
    \E\left[\ltwo{v}^2 \mid X\right]\right].
\end{equation*}
Now consider the inner expectation. Applying the Cauchy-Schwarz
inequality as in the proof of the
bound~\eqref{claim:detruncated-err-small}, we have
\begin{equation*}
  \ltwo{v}^2
  = \sum_{i=1}^n v_i^2
  \le \sum_{i=1}^n \bigg(\sum_{j > d} \frac{\coorderr_j^2}{\eigenvalue_j}
  \bigg) \bigg(\sum_{j > d} \eigenvalue_j \basis_j^2(X_i)\bigg).
\end{equation*}
Notably, the second term is $X$-measureable, and
the first is bounded by $\hnorms{\detruncate{\esterr}}^2
\le \hnorm{\esterr}^2$. We thus obtain
\begin{equation}\label{eqn:err-small-noise-expansion-2}
  \E\left[\ltwo{\eigmat^{1/2} \basismat^T v}^2\right]
  \le \sum_{i=1}^\numobs \sum_{l=1}^d \eigenvalue_\ell
  \E\bigg[\ltwo{\basismat_\ell}^2
    \bigg(\sum_{j > d} \eigenvalue_j \basis_j^2(X_i)\bigg)
    \E[\hnorm{\esterr}^2 \mid X]\bigg].
\end{equation}
Lemma~\ref{lemma:diff-bound} provides the bound $2 \stddev^2 / \lambda + 4
\hnorm{\fopt}^2$ on the final (inner) expectation.

The remainder of the argument proceeds precisely as in the
bound~\eqref{claim:detruncated-err-small}. We have
\begin{equation*}
  \E[\ltwo{\basismat_\ell}^2 \basis_j(X_i)^2]
  \le \numobs \momentbound^4
\end{equation*}
by the moment assumptions on $\basis_j$, and thus
\begin{equation*}
  \E\left[\ltwo{\eigmat^{1/2} \basismat^T v}^2\right]
  \le \sum_{l = 1}^d \sum_{j > d}
  \eigenvalue_\ell \eigenvalue_j \numobs^2 \momentbound^4
  \left(\frac{2 \stddev^2}{\lambda} + 4 \hnorm{\fopt}^2\right)
  \le \numobs^2 \momentbound^4 \tailsum
  \tr(K) \left(\frac{2 \stddev^2}{\lambda} + 4 \hnorm{\fopt}^2\right).
\end{equation*}
Dividing by $\lambda \numobs^2$ completes the proof.

% Local Variables:
% TeX-master: "kernel-regression"
% End:

%% file: oracle-variance-proof.tex
\section{Proof of Lemma~\ref{lemma:oracle-variance-bound}}
\label{appendix:proof-lemma-oracle-variance}

As before, we let $\Data \defeq \{x_1, \ldots, x_\numobs\}$ denote the
collection of design points.  We begin with some useful
bounds on  $\hnorm{\foptimalreg}$ and $\hnorm{\esterr}$.

\begin{lemma}
  \label{lemma:new-pre-bound}
  Under Assumptions~\ref{assumption:kernel}
  and~\ref{assumption:response-moment-bound}, we have
  \begin{align}
    \label{eqn:new-diff-moment-bounds}
    \E\left[(\E[\hnorm{\esterr}^2 \mid \Data])^2\right]
    \leq \difffourth ~~~
    \mbox{and} ~~~ \E[\hnorm{\esterr}^2] \leq \diffsquare,
  \end{align}
  where
  \begin{equation}
    \label{eqn:def-diffsquare}
    \diffone \defeq \sqrt[4]{32\hnorms{\foptimalreg}^4 +
      8\altvarreg^4/{\lambda^2}}.
  \end{equation}
\end{lemma}
\noindent See Section~\ref{sec:proof-new-pre-bound} for the proof of
this claim.

This proof follows an outline similar to that of
Lemma~\ref{lemma:variance-bound}.  As usual, we let $\coorderr \in
\ell_2(\N)$ be the expansion of the error $\esterr$ in the basis
$\{\basis_j\}$, so that $\esterr = \sum_{j=1}^\infty \coorderr_j
\basis_j$, and we choose $d \in \N$ and define the truncated vectors
$\truncate{\esterr} \defeq \sum_{j = 1}^d \coorderr_j \basis_j$ and
$\detruncate{\esterr} \defeq \esterr - \truncate{\esterr} = \sum_{j >
  d} \coorderr_j \basis_j$.  As usual, we have the decomposition
$\E[\ltwo{\esterr}^2] = \E[\ltwos{\truncate{\coorderr}}^2] +
\E[\ltwos{\detruncate{\coorderr}}^2]$. Recall the
definition~\eqref{eqn:def-diffsquare} of the constant $\diffone =
\sqrt[4]{32 \hnorms{\foptimalreg}^4 + 8 \altvarreg^4 / \lambda^2}$.
As in our deduction of
inequalities~\eqref{eqn:detruncated-part-is-small-with-noise},
Lemma~\ref{lemma:new-pre-bound} implies that
\mbox{$\E[\ltwos{\detruncate{\esterr}}^2] \le \eigenvalue_{d + 1}
  \E[\hnorms{\detruncate{\esterr}}^2] \le \eigenvalue_{d + 1}
  \diffsquare$.}

The remainder of the proof is devoted to bounding
$\E[\ltwos{\truncate{\coorderr}}^2]$. We use identical notation to
that in our proof of Lemma~\ref{lemma:variance-bound}, which we recap
for reference (see also Table~\ref{table:notation}).  We define the
tail error vector $v \in \R^\numobs$ by $v_i = \sum_{j > d}
\coorderr_j \basis_j(x_i)$, $i \in [n]$, and recall the definitions of
the eigenvalue matrix $\eigmat = \diag(\eigenvalue_1, \ldots,
\eigenvalue_d) \in \R^{d \times d}$ and basis matrix $\basismat$ with
$\basismat_{ij} = \basis_j(x_i)$.  We use $Q = (I + \lambda
\eigmat^{-1})^{1/2}$ for shorthand, and we let $\event$ be the event that
\begin{equation*}
  \matrixnorm{Q^{-1}((1/\numobs) \basismat^T\basismat - I)Q^{-1}} \le
  1/2.
\end{equation*}
Writing $\foptimalreg = \sum_{j=1}^\infty \theta_j \basis_j$, we
define the alternate noise vector $\noise'_i = Y_i -
\foptimalreg(x_i)$.  Using this notation, mirroring the proof of
Lemma~\ref{lemma:variance-bound} yields
\begin{align}
  \E[\ltwos{\truncate{\esterr}}^2] \leq \E[\ltwos{Q \truncate{\coorderr}}^2]
  \leq 4 \E\left[\ltwo{\lambda Q^{-1}\eigmat^{-1}\truncate{\theta}
      + \frac{1}{\nsample} Q^{-1}\basismat^T v
      - \frac{1}{\nsample} Q^{-1}\basismat^T \noise'}^2\right]
  + \P(\event^c)\diffsquare,
  \label{eqn:oracle-bounding-truncation-with-noise}
\end{align}
which is an analogue of
equation~\eqref{eqn:bounding-truncation-with-noise}.  The bound
bound~\eqref{claim:bias-probability-small} controls the probability
$\P(\event^c)$, so it remains to control the first term in the
expression~\eqref{eqn:oracle-bounding-truncation-with-noise}.  We
first rewrite the expression within the norm as
\begin{align*}
  (\lambda-\lambdabase) Q^{-1}\eigmat^{-1}\truncate{\theta} +
  \frac{1}{\nsample} Q^{-1}\basismat^T v - \left(\frac{1}{\nsample}
  Q^{-1}\basismat^T \noise' -
  \lambdabase Q^{-1}\eigmat^{-1}\truncate{\theta} \right)
\end{align*}
The following lemma provides bounds on the first two terms:
\begin{lemma}
  \label{lemma:oracle-bound-bias-components-1}
  The following bounds hold:
  \begin{subequations}
    \begin{align}
      \ltwo{(\lambdabase-\lambda) Q^{-1}\eigmat^{-1} \truncate{\theta}}^2 &
      \leq \frac{(\lambdabase-\lambda)^2\hnorm{\foptimalreg}^2}{\lambda},
      \label{eqn:oracle-bound-bias-from-regularizer} \\
      \E\left[\ltwo{\frac{1}{n} Q^{-1} \basismat^T v}^2\right] & \leq
      \frac{\momentbound^4 \diffsquare \tr(K) \tailsum}{\lambda},
      \label{eqn:oracle-bound-detruncated-err}
    \end{align}
  \end{subequations}
\end{lemma}
For the third term, we make the following claim.
\begin{lemma}
  \label{LemFinal}
  Under Assumptions~\ref{assumption:kernel}
  and~\ref{assumption:response-moment-bound}, we have
  \begin{align}
    \label{claim:oracle-noise}
    \E \left[\ltwo{\frac{1}{\nsample} Q^{-1}\basismat^T \noise' -
        \lambdabase Q^{-1} \eigmat^{-1}\truncate{\theta}}^2\right] & \leq
    \frac{\fulltracelambda\momentbound^2\altvarreg^2 }{\numobs}.
  \end{align}
\end{lemma}

Deferring the proof of the two lemma to Section~\ref{sec:proof-oracle-bias-components-1} and
Section~\ref{sec:proof-oracle-noise}, we apply the inequality
$(a+b+c)^2\leq 4a^2+4b^2+2c^2$ to
inequality~\eqref{eqn:oracle-bounding-truncation-with-noise}, and we
have
\begin{align}
  & \E[\ltwos{\truncate{\esterr}}^2] - \P(\event^c)
    \diffsquare \leq \E[\ltwos{Q\truncate{\coorderr}}^2] - \P(\event^c)\diffsquare \nonumber \\ 
  & \qquad \le 16 \E\left[\ltwo{(\lambda - \lambdabase)
      Q^{-1} \eigmat^{-1} \truncate{\theta}}^2\right] +
  \frac{16}{\numobs^2} \E\left[\ltwo{Q^{-1} \basismat^T v}^2\right] +
  \frac{8}{\numobs^2} \E\left[ \ltwo{Q^{-1} \basismat^T \noise' -
      \lambdabase Q^{-1} \eigmat^{-1} \truncate{\theta}}^2\right] \nonumber \\ 
  & \qquad
  \leq \frac{16(\lambdabase-\lambda)^2\hnorm{\foptimalreg}^2}{\lambda} +
  \frac{16\momentbound^4 \diffsquare \tr(K) \tailsum}{\lambda} +
  \frac{8 \fulltracelambda \momentbound^2\altvarreg^2 }{\numobs},\label{eqn:oracle-variance-bound-useful-for-bias-bound}
\end{align}
where we have applied the
bounds~\eqref{eqn:oracle-bound-bias-from-regularizer}
and~\eqref{eqn:oracle-bound-detruncated-err} from
Lemma~\ref{lemma:oracle-bound-bias-components} and the
bound~\eqref{claim:oracle-noise} from Lemma~\ref{LemFinal}.  Applying the
bound~\eqref{claim:bias-probability-small} to control $\P(\event^c)$
and recalling that $\E[\ltwo{\detruncate{\esterr}}^2] \le
\eigenvalue_{d+1} \diffsquare$ completes the proof.

%%%%%%%%%%%%%%%%%%%%%%%%%%%%%%%%%%%%%%%%%%%%%%%%%%%%%%%%%%%%%%%%%%%%%%%%%

\subsection{Proof of Lemma~\ref{lemma:new-pre-bound}}
\label{sec:proof-new-pre-bound}

Recall that $\funcapprox$ minimizes the empirical objective. Consequently,
\begin{align*}
  \lambda \E[\hnorms{\funcapprox}^2\mid \Data] & \leq \E\left[
    \frac{1}{\numobs}\sum_{i=1}^\numobs (\funcapprox(x_i)-Y_i)^2 +
    \lambda\hnorms{\funcapprox}^2 \mid \Data \right]\\ & \leq
  \frac{1}{\numobs}\sum_{i=1}^\numobs \E[(\foptimalreg(x_i)-Y_i)^2
    \mid x_i] + \lambda\hnorms{\foptimalreg}^2 =
  \frac{1}{\numobs}\sum_{i=1}^\numobs \varreg^2(x_i) +
  \lambda\hnorms{\foptimalreg}^2
\end{align*}
The triangle inequality
immediately gives us the upper bound
\begin{align*}
  % \label{eqn:new-bounding-error-hilbert-norm}
  \E[\hnorms{\esterr}^2\mid \Data]
  \leq  2\hnorms{\foptimalreg}^2 +
  \E[2\hnorms{\funcapprox}^2\mid \Data]
  \leq \frac{2}{\lambda\numobs}\sum_{i=1}^\numobs \varreg^2(x_i)
  + 4\hnorms{\foptimalreg}^2.
\end{align*}
Since $(a + b)^2 \le 2a^2 + 2b^2$, convexity yields
\begin{align*}
  \E[(\E[\hnorm{\esterr}^2\mid \Data])^2]
  & \le \E\left[ \left(\frac{2}{\lambda\numobs}\sum_{i=1}^\numobs
    \varreg^2(X_i) + 4\hnorms{\foptimalreg}^2 \right)^2\right]\\
   &\leq \frac{8}{\lambda^2\numobs}\sum_{i=1}^\numobs
  \E[\varreg^4(X_i)] + 32\hnorms{\foptimalreg}^4
  = 32\hnorms{\foptimalreg}^4 + \frac{8\altvarreg^4}{\lambda^2}.
\end{align*}
This completes the proof of the first of the
inequalities~\eqref{eqn:new-diff-moment-bounds}.  The second of the
inequalities~\eqref{eqn:new-diff-moment-bounds}
follows from the first by Jensen's inequality.

%%%%%%%%%%%%%%%%%%%%%%%%%%%%%%%%%%%%%%%%%%%%%%%%%%%%%%%%%%%%%%%%%%%%%%%

\subsection{Proof of Lemma~\ref{lemma:oracle-bound-bias-components-1}}
\label{sec:proof-oracle-bias-components-1}

Our previous bound~\eqref{claim:truncated-fopt-small} immediately
implies inequality~\eqref{eqn:oracle-bound-bias-from-regularizer}.  To
prove the second upper bound, we follow the proof of the
bound~\eqref{claim:detruncated-err-small-noise}. From
inequalities~\eqref{eqn:err-small-noise-expansion-1}
and~\eqref{eqn:err-small-noise-expansion-2}, we obtain that
\begin{equation}
  \label{eqn:new-detruncated-err-small-noise-1}
  \ltwo{(1 / \numobs) Q^{-1} \basismat^T v}^2
  \le \frac{1}{\lambda \numobs^2}
  \sum_{i=1}^\numobs \sum_{l=1}^d \sum_{j > d} \eigenvalue_\ell \eigenvalue_j
  \E\left[\ltwo{\basismat_\ell}^2\basis_j^2(X_i)
    \E[\hnorm{\esterr}^2 \mid \{X_i\}_{i=1}^n]\right].
\end{equation}
Applying H\"older's inequality yields
\begin{align*}
  \E \left[\ltwo{\basismat_\ell}^2\basis_j^2(X_i) \E[\hnorm{\esterr}^2
      \mid \{X_i\}_{i=1}^n]\right] \leq
  \sqrt{\E[\ltwo{\basismat_\ell}^4\basis_j^4(X_i)]}
  \sqrt{\E[(\E[\hnorm{\esterr}^2 \mid \{X_i\}_{i=1}^n])^2]}.
\end{align*}
Note that Lemma~\ref{lemma:new-pre-bound} provides the bound
$\difffourth$ on the final expectation. By definition of
$\basismat_\ell$, we find that
\begin{align*}
  \E[\ltwo{\basismat_\ell}^4\basis_j^4(x_i)]
  = \E\left[ \left(\sum_{k=1}^\numobs \basis_\ell^2(x_k) \right)^2
    \basis_j^4(x_i)\right]
  \leq \numobs^2 \E\left[ \frac{1}{2}
    \left(\basis_\ell^8(x_1) + \basis_j^8(x_1)\right) \right]
  \leq \numobs^2\momentbound^8,
\end{align*}
where we have used Assumption~\ref{assumption:kernel} with moment
$2\moment \ge 8$, or equivalently $\moment \ge 4$. Thus
\begin{align}
  \label{eqn:new-detruncated-err-small-noise-2}
  \E \left[ \ltwo{\basismat_\ell}^2\basis_j^2(X_i) \E[\hnorm{\esterr}^2
      \mid \{X_i\}_{i=1}^n] \right]
  & \leq \numobs \momentbound^4 \diffsquare.
\end{align}
Combining inequalities~\eqref{eqn:new-detruncated-err-small-noise-1}
and~\eqref{eqn:new-detruncated-err-small-noise-2} yields the
bound~\eqref{eqn:oracle-bound-detruncated-err}.

%%%%%%%%%%%%%%%%%%%%%%%%%%%%%%%%%%%%%%%%%%%%%%%%%%%%%%%%%%%%%%%%%%%%%%%

\subsection{Proof of Lemma~\ref{LemFinal}}
\label{sec:proof-oracle-noise}

Using the fact that $Q$ and $\eigmat$ are diagonal, we have
\begin{equation}
  \label{eqn:oracle-expand-noise}
  \E\left[\ltwo{\frac{1}{\numobs} Q^{-1}\basismat^T \noise'
      - \lambdabase Q^{-1}\eigmat^{-1}\truncate{\theta}}^2\right]
  = \sum_{j=1}^d Q_{jj}^{-2} \E\left[
    \left(\frac{1}{\numobs} \sum_{i=1}^\numobs \basis_j(X_i) \noise_i'
    - \frac{\lambdabase \theta_j}{\eigenvalue_j}\right)^2\right].
\end{equation}
\frechet differentiability and the fact that $\foptimalreg$ is the global
minimizer of the regularized regression problem imply that
\begin{equation*}
  \E[\eval_{X_i} \noise_i'] + \lambdabase \foptimalreg
  = \E\left[\eval_X \left(\<\eval_X, \foptimalreg\> - y\right)\right]
  + \lambdabase \foptimalreg = 0.
\end{equation*}
Taking the (Hilbert) inner product of the preceding display with the
basis function $\basis_j$, we get
\begin{align}
  \E\left[ \basis_j(X_i)\noise'_i - \frac{\lambdabase\theta_j}{\eigenvalue_j}
    \right] = 0.
  \label{eqn:oracle-mean-zero-noise}
\end{align}
Combining the equalities~\eqref{eqn:oracle-expand-noise}
and~\eqref{eqn:oracle-mean-zero-noise} and using the i.i.d.\ nature of
$\Data$ leads to
\begin{align}
  \E\left[\ltwo{\frac{1}{\numobs} Q^{-1}\basismat^T \noise' -
      \lambdabase Q^{-1}\eigmat^{-1}\truncate{\theta}}^2\right] & =
  \sum_{j=1}^d Q_{jj}^{-2} \var\left( \frac{1}{\numobs}
  \sum_{i=1}^\numobs \basis_j(X_i) \noise_i' - \frac{\lambdabase
    \theta_j}{\eigenvalue_j}\right) \nonumber \\ & = \frac{1}{\numobs}
  \sum_{j=1}^d Q_{jj}^{-2} \var\left( \basis_j(X_1) \noise_1'\right).
  \label{eqn:oracle-truncated-variance}
\end{align}

Using the elementary inequality $\var(Z) \le \E[Z^2]$ for any random variable
$Z$, we have from H\"older's inequality that
\begin{equation*}
  \var(\basis_j(X_1) \noise_1')
  \le \E[\basis_j(X_1)^2 (\noise_1')^2]
  \le \sqrt{\E[\basis_j(X_1)^4] \E[\varreg^4(X_1)]}
  \le \sqrt{\momentbound^4}
  \sqrt{\altvarreg^4},
\end{equation*}
where we used
Assumption~\ref{assumption:response-moment-bound} to upper bound the
fourth moment $\E[\varreg^4(X_1)]$.
Using the fact that $Q_{jj}^{-1} \le 1$, we obtain the following
upper bound on the quantity~\eqref{eqn:oracle-truncated-variance}:
\begin{align*}
  \frac{1}{\numobs}
  \sum_{j=1}^d Q_{jj}^{-2} \var(\basis_j(X_1) \noise_1')
  & =
  \frac{1}{\numobs}
  \sum_{j=1}^d \frac{\var(\basis_j(X_1) \noise_1')}{1 + \lambda/\mu_j }
  \le \frac{\fulltracelambda\momentbound^2\altvarreg^2}{\numobs},
\end{align*}
which establishes the claim.

% Local Variables:
% TeX-master: "kernel-regression"
% End:

%% file: oracle-bias-proof.tex
\section{Proof of Lemma~\ref{lemma:oracle-norm-expectation-bound}}
\label{appendix:proof-lemma-oracle-expectation}

At a high-level, the proof is similar to
that of Lemma~\ref{lemma:norm-expectation-moments}, but we take care
since the errors $\foptimalreg(x) - y$ are not conditionally mean-zero
(or of conditionally bounded variance).  
Recalling our notation of $\eval_x$ as the RKHS evaluator for $x$, we
have by assumption that $\fapprox$ minimizes the empirical
objective~\eqref{eqn:local-empirical-objective}.  As in our derivation
of equality~\eqref{eqn:gradient-optimality}, the \frechet
differentiability of this objective implies the first-order optimality
condition
\begin{align}
  \label{eqn:oracle-first-gradient}
  \frac{1}{\numobs}\sum_{i=1}^\numobs \eval_{x_i} \<\eval_{x_i},
  \esterr \> + \frac{1}{\numobs}\sum_{i=1}^\numobs
  (\eval_{x_i}\<\eval_{x_i}, \foptimalreg\> - y_i) + \lambda \esterr +
  \lambda\foptimalreg = 0,
\end{align}
where $\esterr \defeq \fapprox - \foptimalreg$.  In addition, the
optimality of $\foptimalreg$ implies that
\mbox{$\E[\eval_{x_i}(\langle \eval_{x_i}, \foptimalreg \rangle -
    y_i)] + \lambdabase\foptimalreg=0$.} Using this in
equality~\eqref{eqn:oracle-first-gradient}, we take expectations with
respect to $\{x_i, y_i\}$ to obtain
\begin{align*}
  \E\bigg[\frac{1}{\numobs}\sum_{i=1}^\numobs \eval_{X_i}
    \<\eval_{X_i}, \esterr \>
    + \lambda \esterr\bigg]
  + (\lambda-\lambdabase) \foptimalreg = 0.
\end{align*}
Recalling the definition of the sample covariance operator
$\outprodmat \defeq \frac{1}{\numobs} \sum_{i=1}^\numobs \eval_{x_i}
\otimes \eval_{x_i}$, we arrive at
\begin{equation}
  \label{eqn:oracle-gradient-after-e}
  \E[(\outprodmat + \lambda I) \esterr] = (\lambdabase-\lambda) \foptimalreg,
\end{equation}
which is the analogue of our earlier
equality~\eqref{eqn:opt-gradient-equation}.

We now proceed via a truncation argument similar to that used in our
proofs of Lemmas~\ref{lemma:norm-expectation-moments}
and~\ref{lemma:variance-bound}.  Let $\coorderr \in
\ell_2(\N)$ be the expansion of the error $\esterr$ in the basis
$\{\basis_j\}$, so that $\esterr = \sum_{j=1}^\infty \coorderr_j
\basis_j$.  For a fixed (arbitrary) $d \in \N$, define
\begin{equation*}
  \truncate{\esterr} \defeq \sum_{j = 1}^d \coorderr_j \basis_j ~~
  \mbox{and} ~~ \detruncate{\esterr} \defeq \esterr -
  \truncate{\esterr} = \sum_{j > d} \coorderr_j \basis_j,
\end{equation*}
and note that $\ltwo{\E[\esterr]}^2 = \ltwos{\E[\truncate{\esterr}]}^2
+ \ltwos{\E[\detruncate{\esterr}]}^2$. By
Lemma~\ref{lemma:new-pre-bound}, the second term is controlled by
\begin{equation}
  \label{eqn:oracle-detruncated-part-is-small-with-noise}
  \ltwos{E[\detruncate{\esterr}]}^2 \leq \E[\ltwos{\detruncate{\esterr}}^2]
  = \sum_{j > d} \E[\coorderr_j^2]
  \le \sum_{j > d} \frac{\eigenvalue_{d + 1}}{\eigenvalue_j}
  \E[\coorderr_j^2]
  = \eigenvalue_{d + 1} \E[\hnorms{\detruncate{\esterr}}^2]
  \le \eigenvalue_{d + 1}\diffsquare.
\end{equation}

The remainder of the proof is devoted to bounding
$\ltwos{\E[\truncate{\esterr}]}^2$.  Let $\foptimalreg$ have the
expansion $(\theta_1, \theta_2, \ldots)$ in the basis $\{\basis_j\}$.
Recall (as in Lemmas~\ref{lemma:norm-expectation-moments}
and~\ref{lemma:variance-bound}) the definition of the matrix
$\basismat \in \R^{\numobs \times d}$ by its coordinates
$\basismat_{ij} = \basis_j(x_i)$, the diagonal matrix $\eigmat =
\diag(\eigenvalue_1, \ldots, \eigenvalue_d) \succ 0 \in \R^{d \times
  d}$, and the tail error vector $v \in \R^\numobs$ by $v_i = \sum_{j
  > d} \coorderr_j \basis_j(x_i) = \detruncate{\esterr}(x_i)$.
Proceeding precisely as in the derivations of
equalities~\eqref{eqn:truncated-opt-gradient-equation}
and~\eqref{eqn:truncated-gradient-v2}, we have the following equality:
\begin{equation}
  \label{eqn:oracle-truncated-opt-gradient-equation}
  \E\left[\left(\frac{1}{\numobs} \basismat^T \basismat + \lambda
    \eigmat^{-1}\right) \truncate{\coorderr}\right] =
  (\lambdabase-\lambda) \eigmat^{-1} \truncate{\theta} -
  \E\left[\frac{1}{n} \basismat^T v\right].
\end{equation}
Recalling the definition of the shorthand matrix $Q = (I + \lambda
\eigmat^{-1})^{1/2}$, with some algebra we have
\begin{equation*}
  Q^{-1}\bigg(\frac{1}{\numobs}\basismat^T \basismat + \lambda \eigmat^{-1}
  \bigg) =
  Q + Q^{-1}\bigg(\frac{1}{\numobs}\basismat^T \basismat - I \bigg),
\end{equation*}
so we can
expand expression~\eqref{eqn:oracle-truncated-opt-gradient-equation}
as
\begin{align*}
  \E\left[Q \truncate{\coorderr} + Q^{-1}\left(\frac{1}{n}
    \basismat \basismat^T - I\right)\truncate{\coorderr}\right]
  & = \E\left[Q^{-1}\left(\frac{1}{\numobs} \basismat^T \basismat + \lambda
    \eigmat^{-1}\right) \truncate{\coorderr}\right] \\
  & =
  (\lambdabase-\lambda) Q^{-1}\eigmat^{-1} \truncate{\theta} -
  \E\left[\frac{1}{n} Q^{-1}\basismat^T v\right],
\end{align*}
or, rewriting,
\begin{equation}
  \label{eqn:oracle-inverted-truncated-opt-gradient}
  \E[Q \truncate{\coorderr}] = (\lambdabase-\lambda) Q^{-1} \eigmat^{-1}
  \truncate{\theta} - \E \left [ \frac{1}{n} Q^{-1} \basismat^T v
    \right] - \E\left[Q^{-1}\left(\frac{1}{n} \basismat^T \basismat - I
    \right ) \truncate{\coorderr} \right].
\end{equation}
Lemma~\ref{lemma:oracle-bound-bias-components-1} provides bounds on the first two terms
on the right-hand-side of
equation~\eqref{eqn:oracle-inverted-truncated-opt-gradient}.
The following lemma provides upper bounds on the third term:
%%%%%%%%%%%%%
\begin{lemma}
  \label{lemma:oracle-bound-bias-components}
  There exists a universal constant $C$ such that
    \begin{align}
      \ltwo{\E\left[Q^{-1} \left(\frac{1}{n} \basismat^T \basismat - I\right)
          \truncate{\coorderr}\right]}^2
      & \leq \frac{C(\momentbound^2 \fulltracelambda \log d)^2}{\numobs}
      \E\left[\ltwos{Q \truncate{\coorderr}}^2\right],
      \label{eqn:oracle-bound-bias-remainder}
    \end{align}
\end{lemma}
\noindent We defer the proof to
Section~\ref{sec:proof-oracle-bias-components}.\\

Applying Lemma~\ref{lemma:oracle-bound-bias-components-1} and Lemma~\ref{lemma:oracle-bound-bias-components} to
equality~\eqref{eqn:oracle-inverted-truncated-opt-gradient} and using the
standard inequality $(a+b+c)^2 \leq 4a^2 + 4b^2 + 2c^2$, we obtain the upper
bound
\begin{align*}
  % \label{eqn:oracle-combine-bias-components}
  \ltwo{\E[\truncate{\esterr}]}^2
  \leq \frac{4(\lambdabase-\lambda)^2\hnorm{\foptimalreg}^2}{\lambda}
  + \frac{4\momentbound^4 \diffsquare \tr(K) \tailsum}{\lambda}
  + \frac{C(\momentbound^2 \fulltracelambda \log d)^2}{\numobs}
  \E\left[\ltwos{Q \truncate{\coorderr}}^2\right]
\end{align*}
for a universal constant $C$. Note that inequality~\eqref{eqn:oracle-variance-bound-useful-for-bias-bound}
provides a sufficiently tight bound on the term $\E\left[\ltwos{Q \truncate{\coorderr}}^2\right]$.
Combined with
inequality~\eqref{eqn:oracle-detruncated-part-is-small-with-noise}, this
completes the proof of Lemma~\ref{lemma:oracle-norm-expectation-bound}. \\

%%%%%%%%%%%%%%%%%%%%%%%%%%%%%%%%%%%%%%%%%%%%%%%%%%%%%%%%%%%%%%%%%%%%%%%%%%

\subsection{Proof of Lemma~\ref{lemma:oracle-bound-bias-components}}
\label{sec:proof-oracle-bias-components}

By using Jensen's inequality
and then applying Cauchy-Schwarz, we find
\begin{align*}
  % \label{eqn:oracle-decompose-compound-remainder}
  \ltwo{\E\left[Q^{-1}\left(\frac{1}{n} \basismat^T \basismat -
      I\right)\truncate{\coorderr}\right]}^2 &\leq
  \left(\E\left[\ltwo{Q^{-1} \left(\frac{1}{n} \basismat^T \basismat -
      I\right)\truncate{\coorderr}}\right]\right)^2\nonumber \\
& \leq \E\left[\matrixnorm{Q^{-1}\left(\frac{1}{n} \basismat^T
      \basismat -
      I\right)Q^{-1}}^2\right]\E\left[\ltwos{Q\truncate{\coorderr}}^2\right].
\end{align*}
The first component of the final product can be controlled by the
matrix moment bound established in the proof of
inequality~\eqref{claim:bias-probability-small}.  In particular,
applying~\eqref{eqn:matrix-moment-bound} with $k=2$ yields a
universal constant $C$ such that
\begin{align*}
  % \label{eqn:oracle-bound-second-moment}
  \E\left[\matrixnorm{Q^{-1}\left(\frac{1}{n} \basismat^T \basismat -
      I \right)Q^{-1}}^2\right] \leq
  \frac{C(\momentbound^2 \fulltracelambda \log d)^2}{\numobs},
\end{align*}
which establishes the claim~\eqref{eqn:oracle-bound-bias-remainder}.

% Local Variables:
% TeX-master: "kernel-regression"
% End:

%% file: kernel-regression.bbl
\begin{thebibliography}{36}
\providecommand{\natexlab}[1]{#1}
\providecommand{\url}[1]{\texttt{#1}}
\expandafter\ifx\csname urlstyle\endcsname\relax
  \providecommand{\doi}[1]{doi: #1}\else
  \providecommand{\doi}{doi: \begingroup \urlstyle{rm}\Url}\fi

\bibitem[Bach(2013)]{Bach2013}
F.~Bach.
\newblock Sharp analysis of low-rank kernel matrix approximations.
\newblock In \emph{Proceedings of the Twenty Sixth Annual Conference on
  Computational Learning Theory}, 2013.

\bibitem[Bartlett et~al.(2005)Bartlett, Bousquet, and Mendelson]{Bartlett2005}
P.~Bartlett, O.~Bousquet, and S.~Mendelson.
\newblock Local {R}ademacher complexities.
\newblock \emph{Annals of Statistics}, 33\penalty0 (4):\penalty0 1497--1537,
  2005.

\bibitem[Berlinet and Thomas-Agnan(2004)]{Berlinet2004}
A.~Berlinet and C.~Thomas-Agnan.
\newblock \emph{Reproducing Kernel Hilbert Spaces in Probability and
  Statistics}.
\newblock Kluwer Academic, 2004.

\bibitem[Bertin-Mahieux et~al.(2011)Bertin-Mahieux, Ellis, Whitman, and
  Lamere]{Bertin2011}
T.~Bertin-Mahieux, D.~P. Ellis, B.~Whitman, and P.~Lamere.
\newblock The million song dataset.
\newblock In \emph{Proceedings of the 12th International Conference on Music
  Information Retrieval (ISMIR)}, 2011.

\bibitem[Birman and Solomjak(1967)]{Birman1967}
M.~Birman and M.~Solomjak.
\newblock Piecewise-polynomial approximations of functions of the classes
  ${W}_p^{\alpha}$.
\newblock \emph{Sbornik: Mathematics}, 2\penalty0 (3):\penalty0 295--317, 1967.

\bibitem[Blanchard and Kr\"amer(2010)]{Blanchard2010}
G.~Blanchard and N.~Kr\"amer.
\newblock Optimal learning rates for kernel conjugate gradient regression.
\newblock In \emph{Advances in Neural Information Processing Systems 24}, 2010.

\bibitem[Caponnetto and {D}e {V}ito(2007)]{Caponnetto2007}
A.~Caponnetto and E.~{D}e {V}ito.
\newblock Optimal rates for the regularized least-squares algorithm.
\newblock \emph{Foundations of {C}omputational {M}athematics}, 7\penalty0
  (3):\penalty0 331--368, 2007.

\bibitem[Chen et~al.(2012)Chen, Gittens, and Tropp]{Chen2012}
R.~Chen, A.~Gittens, and J.~A. Tropp.
\newblock The masked sample covariance estimator: an analysis using matrix
  concentration inequalities.
\newblock \emph{Information and Inference}, to appear, 2012.

\bibitem[Fine and Scheinberg(2002)]{Fine2002}
S.~Fine and K.~Scheinberg.
\newblock Efficient {SVM} training using low-rank kernel representations.
\newblock \emph{Journal of Machine Learning Research}, 2:\penalty0 243--264,
  2002.

\bibitem[Gu(2002)]{Gu2002}
C.~Gu.
\newblock \emph{Smoothing spline ANOVA models}.
\newblock Springer, 2002.

\bibitem[Gyorfi et~al.(2002)Gyorfi, Kohler, Krzyzak, and Walk]{GyorfiEtal}
L.~Gyorfi, M.~Kohler, A.~Krzyzak, and H.~Walk.
\newblock \emph{A Distribution-Free Theory of Nonparametric Regression}.
\newblock Springer Series in Statistics. Springer, 2002.

\bibitem[Hastie et~al.(2001)Hastie, Tibshirani, and Friedman]{Hastie2001}
T.~Hastie, R.~Tibshirani, and J.~Friedman.
\newblock \emph{The Elements of Statistical Learning}.
\newblock Springer, 2001.

\bibitem[Hoerl and Kennard(1970)]{Hoerl70}
A.~E. Hoerl and R.~W. Kennard.
\newblock Ridge regression: {B}iased estimation for nonorthogonal problems.
\newblock \emph{Technometrics}, 12:\penalty0 55--67, 1970.

\bibitem[Hsu et~al.(2012)Hsu, Kakade, and Zhang]{Hsu2012}
D.~Hsu, S.~Kakade, and T.~Zhang.
\newblock Random design analysis of ridge regression.
\newblock In \emph{Proceedings of the 25nd Annual Conference on Learning
  Theory}, 2012.

\bibitem[Kleiner et~al.(2012)Kleiner, Talwalkar, Sarkar, and
  Jordan]{Kleiner2012}
A.~Kleiner, A.~Talwalkar, P.~Sarkar, and M.~Jordan.
\newblock Bootstrapping big data.
\newblock In \emph{Proceedings of the 29th International Conference on Machine
  Learning}, 2012.

\bibitem[Koltchinskii(2006)]{Koltchinskii2006}
V.~Koltchinskii.
\newblock Local {R}ademacher complexities and oracle inequalities in risk
  minimization.
\newblock \emph{Annals of Statistics}, 34\penalty0 (6):\penalty0 2593--2656,
  2006.

\bibitem[Ledoux and Talagrand(1991)]{Ledoux1991}
M.~Ledoux and M.~Talagrand.
\newblock \emph{Probability in Banach Spaces}.
\newblock Springer, 1991.

\bibitem[Luenberger(1969)]{Luenberger1969}
D.~Luenberger.
\newblock \emph{Optimization by Vector Space Methods}.
\newblock Wiley, 1969.

\bibitem[McDonald et~al.(2010)McDonald, Hall, and Mann]{McDonald2010}
R.~McDonald, K.~Hall, and G.~Mann.
\newblock Distributed training strategies for the structured perceptron.
\newblock In \emph{North American Chapter of the Association for Computational
  Linguistics (NAACL)}, 2010.

\bibitem[Mendelson(2002)]{Mendelson02}
S.~Mendelson.
\newblock Geometric parameters of kernel machines.
\newblock In \emph{Proceedings of {COLT}}, pages 29--43, 2002.

\bibitem[Rahimi and Recht(2007)]{Rahimi2007}
A.~Rahimi and B.~Recht.
\newblock Random features for large-scale kernel machines.
\newblock In \emph{Advances in Neural Information Processing Systems 20}, 2007.

\bibitem[Raskutti et~al.(2011)Raskutti, Wainwright, and Yu]{Raskutti2011}
G.~Raskutti, M.~Wainwright, and B.~Yu.
\newblock Early stopping for non-parametric regression: An optimal
  data-dependent stopping rule.
\newblock In \emph{49th Annual Allerton Conference on Communication, Control,
  and Computing}, pages 1318--1325, 2011.

\bibitem[Raskutti et~al.(2012)Raskutti, Wainwright, and Yu]{RasWaiYu12}
G.~Raskutti, M.~J. Wainwright, and B.~Yu.
\newblock Minimax-optimal rates for sparse additive models over kernel classes
  via convex programming.
\newblock \emph{Journal of {M}achine {L}earning {R}esearch}, 12:\penalty0
  389--427, March 2012.

\bibitem[Saunders et~al.(1998)Saunders, Gammerman, and Vovk]{Saunders1998}
C.~Saunders, A.~Gammerman, and V.~Vovk.
\newblock Ridge regression learning algorithm in dual variables.
\newblock In \emph{Proceedings of the 15th International Conference on Machine
  Learning}, pages 515--521. Morgan Kaufmann, 1998.

\bibitem[Sch\"{o}lkopf et~al.(1998)Sch\"{o}lkopf, Smola, and
  M\"{u}ller]{Scholkopf1998b}
B.~Sch\"{o}lkopf, A.~Smola, and K.-R. M\"{u}ller.
\newblock Nonlinear component analysis as a kernel eigenvalue problem.
\newblock \emph{IEEE Transactions on Information Theory}, 10\penalty0
  (5):\penalty0 1299--1319, 1998.

\bibitem[Shawe-Taylor and Cristianini(2004)]{Shawe-Taylor2004}
J.~Shawe-Taylor and N.~Cristianini.
\newblock \emph{Kernel Methods for Pattern Analysis}.
\newblock Cambridge University Press, 2004.

\bibitem[Steinwart et~al.(2009)Steinwart, Hush, and Scovel]{Steinwart2009}
I.~Steinwart, D.~Hush, and C.~Scovel.
\newblock Optimal rates for regularized least squares regression.
\newblock In \emph{Proceedings of the 22nd Annual Conference on Learning
  Theory}, pages 79--93, 2009.

\bibitem[Stone(1982)]{Stone82}
C.~J. Stone.
\newblock Optimal global rates of convergence for non-parametric regression.
\newblock \emph{Annals of {S}tatistics}, 10\penalty0 (4):\penalty0 1040--1053,
  1982.

\bibitem[Tsybakov(2009)]{Tsybakov2009}
A.~B. Tsybakov.
\newblock \emph{Introduction to Nonparametric Estimation}.
\newblock Springer, 2009.

\bibitem[van~de Geer(2000)]{vandeGeer}
S.~van~de Geer.
\newblock \emph{Empirical Processes in M-Estimation}.
\newblock Cambridge University Press, 2000.

\bibitem[Wahba(1990)]{Wahba}
G.~Wahba.
\newblock \emph{Spline models for observational data}.
\newblock CBMS-NSF Regional Conference Series in Applied Mathematics. SIAM,
  Philadelphia, PN, 1990.

\bibitem[Wasserman(2006)]{Wasserman2006}
L.~Wasserman.
\newblock \emph{All of Nonparametric Statistics}.
\newblock Springer, 2006.

\bibitem[Williams and Seeger(2001)]{Williams2001}
C.~Williams and M.~Seeger.
\newblock Using the {N}ystr\"om method to speed up kernel machines.
\newblock \emph{Advances in Neural Information Processing Systems 14}, pages
  682--688, 2001.

\bibitem[Yao et~al.(2007)Yao, Rosasco, and Caponnetto]{Yao2007}
Y.~Yao, L.~Rosasco, and A.~Caponnetto.
\newblock On early stopping in gradient descent learning.
\newblock \emph{Constructive Approximation}, 26\penalty0 (2):\penalty0
  289--315, 2007.

\bibitem[Zhang(2005)]{Zhang2005b}
T.~Zhang.
\newblock Learning bounds for kernel regression using effective data
  dimensionality.
\newblock \emph{Neural Computation}, 17\penalty0 (9):\penalty0 2077--2098,
  2005.

\bibitem[Zhang et~al.(2012)Zhang, Duchi, and Wainwright]{Zhang2012a}
Y.~Zhang, J.~C. Duchi, and M.~J. Wainwright.
\newblock Communication-efficient algorithms for statistical optimization.
\newblock In \emph{Advances in Neural Information Processing Systems 26}, 2012.

\end{thebibliography}
